\documentclass[final]{siamltex}

\usepackage{amsmath, amssymb, mathtools, extarrows, cite, multirow}
\usepackage{tikz-cd}
\newcommand{\mzero}{\mathbf O }
\newcommand{\vzero}{\mathbf 0 }
\newcommand{\hatdSmR}[2]{\widehat{\mathbf d}^{#1}}
\newcommand{\cond}[2]{\textup{({#1}{\scriptsize {#2}})}}

\newtheorem{remark}[theorem]{Remark}
\newtheorem{example}[theorem]{Example}

\allowdisplaybreaks

\title{On the Uniqueness of the Canonical Polyadic Decomposition of  third-order tensors --- Part II: Uniqueness of the overall decomposition
\thanks{Research supported by: (1) Research Council KU Leuven:
GOA-Ambiorics,  GOA-MaNet,  CoE EF/05/006 Optimization in Engineering (OPTEC),
CIF1,  STRT 1/08/23,  (2) F.W.O.: (a) project G.0427.10N,  (b) Research Communities
ICCoS,  ANMMM and MLDM,  (3) the Belgian Federal Science Policy
Office: IUAP P6/04 (DYSCO,  ``Dynamical systems,  control and
optimization'',  2007--2011),  (4) EU: ERNSI.}}

\author{Ignat Domanov\footnotemark[2] \footnotemark[3]
\and
Lieven De Lathauwer\footnotemark[2] \footnotemark[3]}

\begin{document}

\maketitle

\renewcommand{\thefootnote}{\fnsymbol{footnote}}

\footnotetext[2]{Group Science,  Engineering and Technology,  KU Leuven Campus Kortrijk,
Etienne Sabbelaan 53,  8500 Kortrijk,  Belgium,
 ({\tt ignat.domanov,\ lieven.delathauwer@kuleuven-kulak.be}).}
\footnotetext[3]{Department of Electrical Engineering (ESAT),  SCD,  KU Leuven,
Kasteelpark Arenberg 10,  postbus 2440,  B-3001 Heverlee (Leuven),  Belgium.}

\begin{abstract}
Canonical Polyadic (also known as Candecomp/Parafac) Decomposition (CPD) of a higher-order tensor is  decomposition in a minimal number of rank-$1$ tensors. In Part I, we gave an overview of existing results concerning uniqueness and presented new, relaxed, conditions that guarantee uniqueness of one factor matrix. In Part II we use these results for establishing overall CPD uniqueness in cases where none of the factor matrices has full column rank. We obtain uniqueness conditions involving Khatri-Rao products of compound matrices and Kruskal-type conditions.
%
%
We consider both deterministic and generic uniqueness. We also discuss uniqueness of INDSCAL and other constrained polyadic decompositions.
\end{abstract}

\begin{keywords}
Canonical Polyadic Decomposition, Candecomp, Parafac, three-way array, tensor, multilinear algebra, Khatri-Rao product, compound matrix
\end{keywords}

\begin{AMS}
15A69, 15A23
\end{AMS}

\pagestyle{myheadings}
\thispagestyle{plain}
\markboth{IGNAT DOMANOV AND LIEVEN DE LATHAUWER}{On the Uniqueness of the Canonical Polyadic Decomposition --- Part II}

\section{Introduction}
\subsection{Problem statement}\label{Subsection1.1}
Throughout the paper $\mathbb F$ denotes the field of real or complex numbers;
$(\cdot)^*$, $(\cdot)^T$, and $(\cdot)^H$ denote conjugate, transpose, and conjugate transpose, respectively;
$r_{\mathbf A}$,  $\textup{range}(\mathbf A)$, and
 $\textup{ker}(\mathbf A)$ denote the rank, the range, and the null space of a matrix $\mathbf A$, respectively;
$\textup{\text{Diag}}(\mathbf d)$ denotes a square diagonal matrix with the elements of a vector $\mathbf d$ on the main diagonal;
$\text{span}\{\mathbf f_1, \dots, \mathbf f_k\}$ denotes the linear span of the vectors $\mathbf f_1, \dots, \mathbf f_k$;
$\mathbf e_r^R$ denotes the $r$-th vector of the canonical basis of $\mathbb F^R$;
$C_n^k$ denotes the binomial coefficient,  $C_n^k=\frac{n!}{k!(n-k)!}$;
$\mzero_{m\times n}$, $\vzero_m$, and $\mathbf I_n$ are the zero $m\times n$ matrix, the zero $m\times 1$ vector, and the $n\times n$ identity matrix, respectively.

We have the following basic definitions. A third-order tensor $\mathcal T=(t_{ijk})\in\mathbb F^{I\times J\times K}$ is {\em rank-$1$}
if there exist  three nonzero vectors $\mathbf a\in\mathbb F^I$,
$\mathbf b\in\mathbb F^J$ and $\mathbf c\in\mathbb F^K$ such that
$\mathcal T=\mathbf a\circ\mathbf b\circ\mathbf c$,
in  which ``$\circ$'' denotes the {\em outer product}. That is,  $t_{ijk} = a_i b_j c_k$ for all values of the indices.

A {\em Polyadic Decomposition} (PD) of a third-order tensor $\mathcal T\in\mathbb F^{I\times J\times K}$ expresses $\mathcal T$ as a sum of  rank-$1$ terms:
\begin{equation}
\mathcal T=\sum\limits_{r=1}^R\mathbf a_r\circ \mathbf b_r\circ \mathbf c_r,
 \label{eqintro2}
\end{equation}
where
$\mathbf a_r \in \mathbb F^{I}$, $\mathbf b_r \in \mathbb F^{J}$, $\mathbf c_r \in \mathbb F^{K}$ are nonzero vectors.

We call the matrices
$\mathbf A =\left[\begin{matrix}\mathbf a_1&\dots&\mathbf a_R\end{matrix}\right] \in\mathbb F^{I\times R}$, $\mathbf B =\left[\begin{matrix}\mathbf b_1&\dots&\mathbf b_R\end{matrix}\right]\in\mathbb F^{J\times R}$ and $\mathbf C =\left[\begin{matrix}\mathbf c_1&\dots&\mathbf c_R\end{matrix}\right]\in\mathbb F^{K\times R}$
the {\em first, second} and {\em third factor matrix} of $\mathcal T$, respectively. We also write (\ref{eqintro2}) as $\mathcal T=[\mathbf A,\mathbf B,\mathbf C]_R$.

\begin{definition}
The {\em rank} of a tensor $\mathcal T \in\mathbb F^{I\times J\times K}$ is defined
as the minimum number of rank-$1$ tensors in a PD  of  $\mathcal T$ and is denoted by $r_{\mathcal T}$.
\end{definition}
\begin{definition}
A {\em Canonical Polyadic  Decomposition} (CPD) of a third-order tensor $\mathcal T$ expresses $\mathcal T$ as a minimal sum of rank-$1$ terms.
\end{definition}

Note that $\mathcal T=[\mathbf A,\mathbf B,\mathbf C]_R$ is a CPD of $\mathcal T$ if and only if $R=r_{\mathcal T}$.

Let us reshape $\mathcal T$ into a  matrix $\mathbf T\in\mathbb F^{IJ\times K}$ as follows:
the $\left(i,j,k\right)$-th entry of $\mathcal T$ corresponds to
the $\left((i-1)J+j,k\right)$-th entry of $\mathbf T$.
In particular, the rank-1 tensor $\mathbf a\circ\mathbf b\circ\mathbf c$ corresponds to
the rank-1 matrix $(\mathbf a\otimes\mathbf b)\mathbf c^T$, in which
``$\otimes$'' denotes the {\em Kronecker product}.
Thus, \eqref{eqintro2} can be identified  with
\begin{equation}
\mathbf T^{(1)}:=\mathbf T=\sum\limits_{r=1}^R(\mathbf a_r\otimes \mathbf b_r){\mathbf c}_r^T=
[\mathbf a_1\otimes\mathbf b_1\ \cdots\ \mathbf a_R\otimes\mathbf b_R]\mathbf C^T=
(\mathbf A\odot\mathbf B)\mathbf C^T,
\label{eqT_V}
\end{equation}
in which ``$\odot$'' denotes the {\em Khatri-Rao product} or column-wise Kronecker product.
Similarly, one can reshape $\mathbf a\circ\mathbf b\circ\mathbf c$  into any of the matrices
$$
(\mathbf b\otimes\mathbf c)\mathbf a^T,\quad
(\mathbf c\otimes\mathbf a)\mathbf b^T,\quad
(\mathbf a\otimes\mathbf c)\mathbf b^T,\quad
(\mathbf b\otimes\mathbf a)\mathbf c^T,\quad
(\mathbf c\otimes\mathbf b)\mathbf a^T
$$
and obtain the factorizations
\begin{equation}
\mathbf T^{(2)}=(\mathbf B\odot\mathbf C)\mathbf A^T, \qquad \mathbf T^{(3)}=(\mathbf C\odot\mathbf A)\mathbf B^T,
\qquad \mathbf T^{(4)}=(\mathbf A\odot\mathbf C)\mathbf B^T\qquad \text{etc.}
\label{eqT_V2}
\end{equation}
The matrices $\mathbf T^{(1)}$, $\mathbf T^{(2)},\dots$ are called
{\em the matrix representations} or {\em matrix unfoldings} of the tensor $\mathcal T$.

It is clear that in \eqref{eqintro2}--\eqref{eqT_V} the rank-1 terms can be arbitrarily permuted and that vectors within the same rank-1 term can be arbitrarily scaled provided the overall rank-1 term remains the same. {\em The CPD of a tensor is  unique} when it is only subject to these trivial indeterminacies. Formally, we have the following definition.

\begin{definition}\label{Def:1.5}
Let $\mathcal T$ be a tensor of rank $R$. The CPD  of  $\mathcal T$ is {\em essentially unique} if
$\mathcal T=[\mathbf A,\mathbf B,\mathbf C]_R=[\bar{\mathbf A},\bar{\mathbf B},\bar{\mathbf C}]_R$ implies that there exist an $R\times R$ permutation matrix $\mathbf{\Pi}$ and $R\times R$ nonsingular diagonal matrices
${\mathbf \Lambda}_{\mathbf A}$, ${\mathbf \Lambda}_{\mathbf B}$, and ${\mathbf \Lambda}_{\mathbf C}$ such that
$$
\bar{\mathbf A}=\mathbf A\mathbf{\Pi}{\mathbf \Lambda}_{\mathbf A},\quad
\bar{\mathbf B}=\mathbf B\mathbf{\Pi}{\mathbf \Lambda}_{\mathbf B},\quad
\bar{\mathbf C}=\mathbf C\mathbf{\Pi}{\mathbf \Lambda}_{\mathbf C},\quad
{\mathbf \Lambda}_{\mathbf A}{\mathbf \Lambda}_{\mathbf B}{\mathbf \Lambda}_{\mathbf C}=\mathbf I_R.
$$
\end{definition}
PDs can also be partially unique. That is, a factor matrix may be essentially unique without the overall PD being essentially unique.
We will resort to the following definition.
\begin{definition}\label{def1.6}
Let $\mathcal T$ be a tensor of rank $R$. The  first (resp. second or third) factor matrix of  $\mathcal T$ is {\em essentially unique}
if $\mathcal T=[\mathbf A,\mathbf B,\mathbf C]_R=[\bar{\mathbf A},\bar{\mathbf B},\bar{\mathbf C}]_R$ implies that there exist an $R\times R$ permutation matrix
$\mathbf{\Pi}$ and an $R\times R$ nonsingular diagonal matrix ${\mathbf \Lambda}_{\mathbf A}$ (resp. ${\mathbf \Lambda}_{\mathbf B}$ or ${\mathbf \Lambda}_{\mathbf C}$) such that
$$
\bar{\mathbf A}=\mathbf A\mathbf{\Pi}{\mathbf \Lambda}_{\mathbf A} \quad (\text{resp.} \quad
\bar{\mathbf B}=\mathbf B\mathbf{\Pi}{\mathbf \Lambda}_{\mathbf B} \quad \ \text{ or }\ \
\bar{\mathbf C}=\mathbf C\mathbf{\Pi}{\mathbf \Lambda}_{\mathbf C}).
$$
\end{definition}
For  brevity, in the sequel we drop the term ``essential'', both when it concerns the uniqueness of the overall CPD and when it concerns the uniqueness of one factor matrix.

In this paper we present both deterministic and generic uniqueness results. Deterministic conditions concern one particular PD $\mathcal T=[\mathbf A,\mathbf B,\mathbf C]_R$. For generic uniqueness we resort to the following definitions.

\begin{definition}\label{def1.7}
Let $\mu$ be the Lebesgue measure on $\mathbb F^{(I+J+K)R}$.
The CPD  of an $I\times J\times K$ tensor of rank $R$ is $generically$ $unique$ if
$$
\mu\{(\mathbf A,\mathbf B,\mathbf C):\ \text{the CPD }  \text{of the tensor }\ [\mathbf A,\mathbf B,\mathbf C]_R
\text{ is not unique }\}=0.
$$
\end{definition}
\begin{definition}\label{def1.8}
Let $\mu$ be the Lebesgue measure on $\mathbb F^{(I+J+K)R}$.
The  first (resp. second or third) factor matrix of an $I\times J\times K$ tensor  of rank $R$ is $generically$ $unique$ if
\begin{equation*}
\begin{split}
\mu\left\{ (\mathbf A,\mathbf B,\mathbf C):\ \right.&\text{the first (resp. second or third) factor matrix } \\
  &\left.\text{of the tensor}\ [\mathbf A,\mathbf B,\mathbf C]_R
\text{ is not unique}\right\}=0.
\end{split}
\end{equation*}
\end{definition}

Let the matrices $\mathbf A\in\mathbb F^{I\times R}$, $\mathbf B\in\mathbb F^{J\times R}$ and $\mathbf C\in\mathbb F^{K\times R}$ be randomly sampled from a continuous distribution.
Generic uniqueness then means uniqueness that holds with probability one.

\subsection{Literature overview}\label{Subsection1.3}
  We refer to the overview papers  \cite{KoldaReview,2009Comonetall,Lieven_ISPA} and the references therein for background,  applications and algorithms for CPD.
Here, we focus on results concerning  uniqueness of the CPD.
\subsubsection{Deterministic conditions}
We refer to \cite[Subsection 1.2]{Part I}  for a
detailed overview of deterministic conditions. Here we just recall
three Kruskal theorems and new results from \cite{Part I} that concern the  uniqueness of one factor matrix.
To present Kruskal's theorem we recall the definition of $k$-rank.
\begin{definition}\label{defkrank}
The $k$-rank  of a matrix $\mathbf A$ is the largest number $k_{\mathbf A}$ such that
every subset of $k_{\mathbf A}$ columns of the matrix $\mathbf A$ is  linearly independent.
\end{definition}

Kruskal's theorem states the following.
\begin{theorem}\cite[Theorem 4a, p. 123]{Kruskal1977}\label{theoremKruskal}
Let $\mathcal T=[\mathbf A,\mathbf B,\mathbf C]_R$  and let
\begin{equation}
k_{\mathbf A}+k_{\mathbf B}+k_{\mathbf C}\geq 2R+2.
 \label{Kruskal}
 \end{equation}
Then $r_{\mathcal T}=R$ and the CPD of  $\mathcal T=[\mathbf A,\mathbf B,\mathbf C]_R$ is  unique.
 \end{theorem}

Kruskal also obtained the  following more general results which are less known.

\begin{theorem}\cite[Theorem 4b, p. 123]{Kruskal1977}\label{theoremKruskalnew1} (see also Corollary \ref{Proptwokkk} below)
Let $\mathcal T=[\mathbf A,\mathbf B,\mathbf C]_R$  and let
\begin{align*}
\ \ \min(k_{\mathbf A},k_{\mathbf C})+r_{\mathbf B}&\geq R+2,\\
\ \ \min(k_{\mathbf A},k_{\mathbf B})+r_{\mathbf C}&\geq R+2,\\
r_{\mathbf A}+r_{\mathbf  B}+r_{\mathbf C}&\geq 2R+2+\min(r_{\mathbf A}-k_{\mathbf A},r_{\mathbf B}-k_{\mathbf B}),\\
r_{\mathbf A}+r_{\mathbf  B}+r_{\mathbf C}&\geq 2R+2+\min(r_{\mathbf A}-k_{\mathbf A},r_{\mathbf C}-k_{\mathbf C}).
\end{align*}
Then $r_{\mathcal T}=R$ and the CPD of  $\mathcal T=[\mathbf A,\mathbf B,\mathbf C]_R$ is  unique.
\end{theorem}

Let the matrices $\mathbf A$ and $\mathbf B$ have $R$ columns. Let  $\tilde{\mathbf A}$ be any set of columns of $\mathbf A$, let $\tilde{\mathbf B}$ be the corresponding set of columns of $\mathbf B$, and define
\begin{equation*}
H_{\mathbf A\mathbf B}(\delta):=\min\limits_{card(\tilde{\mathbf A})=\delta}\left[r_{\tilde{\mathbf A}}+r_{\tilde{\mathbf B}}-\delta\right]\quad\text{for}\quad
\delta=1,2,\dots,R.
\end{equation*}
We will say that condition $(\textup{H{\scriptsize m}})$
holds for the matrices $\mathbf A$ and $\mathbf B$ if
\begin{equation}
H_{\mathbf A\mathbf B}(\delta)\geq\min(\delta,m)\quad\text{for}\quad
\delta=1,2,\dots,R.\tag{H{\scriptsize m}}
\end{equation}
The following Theorem is the strongest result about uniqueness from \cite{Kruskal1977}.
\begin{theorem}\cite[Theorem 4e, p. 125]{Kruskal1977}\label{theoremKruskalnew2}(see also Corollary \ref{Prop5.4HmHm} below)
Let $\mathcal T=[\mathbf A,\mathbf B,\mathbf C]_R$  and let
$m_{\mathbf B}:=R-r_{\mathbf B}+2$, $m_{\mathbf C}:=R-r_{\mathbf C}+2$.
Assume that
\begin{romannum}
\item $(\textup{H{\scriptsize 1}})$ holds for $\mathbf B$ and $\mathbf C$;
\item $\text{\textup{(H{\scriptsize m}}}_{\mathbf B}$$\text{\textup{)}}$ holds for $\mathbf C$ and $\mathbf A$;
\item $\text{\textup{(H{\scriptsize m}}}_{\mathbf C}$$\text{\textup{)}}$ holds for $\mathbf A$ and $\mathbf B$.
\end{romannum}
Then $r_{\mathcal T}=R$ and the CPD of  $\mathcal T=[\mathbf A,\mathbf B,\mathbf C]_R$ is  unique.
\end{theorem}

For the formulation of other results we  recall the definition of compound matrix.
\begin{definition}\label{defcompound} \cite[Definition 2.1  and Example 2.2]{Part I}
The $k$-th compound matrix of $I\times R$ matrix $\mathbf A$ (denoted by $\mathcal C_k(\mathbf A)$) is the $C^k_I\times C^k_R$ matrix containing the determinants of all $k\times k$ submatrices of $\mathbf A$, arranged with the submatrix index sets in lexicographic order.
\end{definition}

 With a vector $\mathbf d=\left[\begin{matrix}d_{1}&\dots& d_R\end{matrix}\right]^T$
we associate the vector
\begin{equation}
\hatdSmR{m}{R}:=
 \left[\begin{matrix}d_{1}\cdots d_m&d_{1}\cdots d_{m-1}d_{m+1}&\dots& d_{R-m+1}\cdots d_R\end{matrix}\right]^T\in\mathbb F^{C^m_R},
  \label{eqd_big_product}
\end{equation}
whose entries are all products $d_{i_1}\cdots d_{i_m}$ with $1\leq i_1<\dots<i_m\leq R$.
 Let us define conditions $(\textup{K{\scriptsize m}})$, $(\textup{C{\scriptsize m}})$, $(\textup{U{\scriptsize m}})$ and $\text{\textup{(W{\scriptsize m})}}$, which
  depend on matrices $\mathbf A\in\mathbb F^{I\times R}$, $\mathbf B\in\mathbb F^{J\times R}$, $\mathbf C\in\mathbb F^{K\times R}$ and an integer parameter $m$:
\begin{align}
&\left\{
  \begin{array}{rl}
    r_{\mathbf A}+k_{\mathbf B}&\geq R+m,\\
    k_{\mathbf A}&\geq m
  \end{array}
\right.
\qquad\text{ or }\qquad
\left\{
  \begin{array}{rl}
      r_{\mathbf B}+k_{\mathbf A}&\geq R+m,\\
      k_{\mathbf B}&\geq m
  \end{array}
\right.
;
\tag{K{\scriptsize m}}
\\
&\ \quad\mathcal C_{m}(\mathbf A)\odot \mathcal C_{m}(\mathbf B)\quad\ \ \text{ has full column rank}\tag{C{\scriptsize m}};
\\
&\begin{cases}
(\mathcal C_{m}(\mathbf A )\odot\mathcal C_m(\mathbf B))\hatdSmR{m}{R}=\vzero,\\
\mathbf d\in\mathbb F^R
\end{cases}\Rightarrow\quad
\hatdSmR{m}{R}=\vzero;
\tag{U{\scriptsize m}}
\\
& \begin{cases}
(\mathcal C_{m}(\mathbf A)\odot \mathcal C_m(\mathbf B))\hatdSmR{m}{R}=\vzero,\\
\mathbf d\in\textup{range}(\mathbf C^T)
\end{cases} \Rightarrow\quad
\hatdSmR{m}{R}=\vzero.\tag{W{\scriptsize m}}
\end{align}
In the sequel, we will for instance say  that  ``condition (\text{U{\scriptsize m}}) holds for
the matrices $\mathbf X$ and $\mathbf Y$'' if
condition (\text{U{\scriptsize m}}) holds for
the matrices $\mathbf A$ and $\mathbf B$ replaced by the matrices $\mathbf X$ and $\mathbf Y$, respectively.
We will  simply write (\text{U{\scriptsize m}}) (resp.  (\text{K{\scriptsize m}}),(\text{H{\scriptsize m}}),(\text{C{\scriptsize m}}) or (\text{W{\scriptsize m}}))
when no confusion is possible.

It is known that conditions $(\text{K{\scriptsize 2}})$, $(\text{C{\scriptsize 2}})$, (\text{U{\scriptsize 2}}) guarantee uniqueness of the CPD with full column rank  in the third mode (see Proposition \ref{K2C2U2unique} below), and that condition $(\text{K{\scriptsize m}})$ guarantees the uniqueness of the  third factor matrix \cite{GuoMironBrieStegeman}, \cite[Theorem 1.12]{Part I}.

In the following Proposition we gather, for later reference, properties of conditions $(\textup{K{\scriptsize m}})$, $(\textup{C{\scriptsize m}})$, $(\textup{U{\scriptsize m}})$ and $\text{\textup{(W{\scriptsize m})}}$ that were established in \cite[\S 2--\S 3]{Part I}. The proofs follow from  properties of compound matrices \cite[Subsection 2.1]{Part I}.

\begin{proposition}\label{Prop:KCUW}
\begin{itemize}
\item[(1)] If $(\textup{K{\scriptsize m}})$ holds, then $(\textup{C{\scriptsize m}})$  and $(\textup{H{\scriptsize m}})$ hold \cite[Lemmas 3.8, 3.9]{Part I};
\item[(2)] if $(\textup{C{\scriptsize m}})$  or $(\textup{H{\scriptsize m}})$  holds, then $(\textup{U{\scriptsize m}})$ holds \cite[Lemmas 3.1, 3.10]{Part I};
\item[(3)] if $(\textup{U{\scriptsize m}})$ holds, then $(\textup{W{\scriptsize m}})$ holds \cite[Lemma 3.3]{Part I};
\item[(4)] if $(\textup{K{\scriptsize m}})$ holds, then $(\textup{K{\scriptsize k}})$ holds for $k\leq m$ \cite[Lemma 3.4]{Part I};
\item[(5)] if $(\textup{H{\scriptsize m}})$ holds, then $(\textup{H{\scriptsize k}})$ holds for $k\leq m$ \cite[Lemma 3.5]{Part I};
\item[(6)] if $(\textup{C{\scriptsize m}})$ holds, then $(\textup{C{\scriptsize k}})$ holds for $k\leq m$ \cite[Lemma 3.6]{Part I};
\item[(7)] if $(\textup{U{\scriptsize m}})$ holds, then $(\textup{U{\scriptsize k}})$ holds for $k\leq m$  \cite[Lemma 3.7]{Part I};
\item[(8)] if $(\textup{W{\scriptsize m}})$ holds and
$\min(k_{\mathbf A}, k_{\mathbf B})\geq m-1$,
 then $(\textup{W{\scriptsize k}})$ holds for $k\leq m$ \cite[Lemma 3.12 ]{Part I};
\item[(9)] if $(\textup{U{\scriptsize m}})$ holds, then $\min(k_{\mathbf A}, k_{\mathbf B})\geq m$ \cite[Lemma 2.8 ]{Part I}.
\end{itemize}
\end{proposition}

The following schemes illustrate  Proposition \ref{Prop:KCUW}:
\begin{equation}\label{maindiagramintro}
\begin{matrix}
\left\{
\begin{matrix}
k_{\mathbf A}\geq m,\\
k_{\mathbf B}\geq m\
\end{matrix}
\right.\\
\quad\\
\quad
\end{matrix}\quad
\begin{matrix}
&(\text{W{\scriptsize m}})       &\
         &(\text{W{\scriptsize m-1}})   &\          &\dots&\          &(\text{W{\scriptsize 2}}) &\  &(\text{W{\scriptsize 1}})\\
&\Uparrow  &\           &\Uparrow  &\          &\dots&\          &\Uparrow  &\           &\Uparrow\\
\Leftarrow&(\text{U{\scriptsize m}})       &\Rightarrow           &(\text{U{\scriptsize m-1}})   &\Rightarrow          &\dots&\Rightarrow          &(\text{U{\scriptsize 2}})       &\Rightarrow  &(\text{U{\scriptsize 1}})\\
&\Uparrow  &\           &\Uparrow  &\          &\dots&\          &\Uparrow  &\           &\Updownarrow\\
&(\text{C{\scriptsize m}})       &\Rightarrow &(\text{C{\scriptsize m-1}})   &\Rightarrow&\dots&\Rightarrow&(\text{C{\scriptsize 2}})       &\Rightarrow &(\text{C{\scriptsize 1}})\\
&\Uparrow  &\           &\Uparrow  &\          &\dots&\          &\Uparrow  &\           &\Uparrow\\
&(\text{K{\scriptsize m}})       &\Rightarrow &(\text{K{\scriptsize m-1}})   &\Rightarrow&\dots&\Rightarrow&(\text{K{\scriptsize 2}})       &\Rightarrow &(\text{K{\scriptsize 1}})
\end{matrix},
\end{equation}
and
\begin{equation}
\text{if }\min(k_{\mathbf A}, k_{\mathbf B})\geq m-1, \text{  then }\quad
 (\text{W{\scriptsize m}})       \Rightarrow
         (\text{W{\scriptsize m-1}})  \Rightarrow          \dots\ \Rightarrow           (\text{W{\scriptsize 2}}) \Rightarrow(\text{W{\scriptsize 1}}).
         \label{eq1.14}
\end{equation}
Scheme
\eqref{maindiagramintro} also remains valid after replacing conditions
$(\textup{C{\scriptsize m}})$,\dots,$(\textup{C{\scriptsize 1}})$ and equivalence
$(\textup{C{\scriptsize 1}})\Leftrightarrow (\textup{U{\scriptsize 1}})$ by
conditions $(\textup{H{\scriptsize m}})$,\dots,$(\textup{H{\scriptsize 1}})$ and implication
$(\textup{H{\scriptsize 1}})\Rightarrow (\textup{U{\scriptsize 1}})$, respectively. One can easily construct examples where $(\textup{C{\scriptsize m}})$ holds but
$(\textup{H{\scriptsize m}})$ does not hold. We do not know examples where $(\textup{H{\scriptsize m}})$ is more relaxed than
$(\textup{C{\scriptsize m}})$.

Deterministic  results concerning the uniqueness of one particular factor matrix were presented in \cite[\S 4]{Part I}. We first have the following proposition.
\begin{proposition}\label{prmostgeneraldis}\cite[Proposition 4.9]{Part I}
Let $\mathbf A\in \mathbb F^{I\times R}$,  $\mathbf B\in \mathbb F^{J\times R}$,  $\mathbf C\in \mathbb F^{K\times R}$, and let
$\mathcal T=[\mathbf A, \mathbf B, \mathbf C]_R$.
Assume that
\begin{romannum}
\item $k_{\mathbf C}\geq 1$;
\item
$m=R-r_{\mathbf C}+2\leq\min(I, J)$;
\item $\mathbf A\odot\mathbf B$ has full column rank;
\item
the triplet of matrices $(\mathbf A,\mathbf B,\mathbf C)$ satisfies conditions $\text{\textup{(W{\scriptsize m})}},\dots,\text{\textup{(W{\scriptsize 1})}}$.
\end{romannum}
Then $r_{\mathcal T}=R$ and the third factor matrix  of  $\mathcal T$  is unique.
\end{proposition}

Combining Propositions \ref{Prop:KCUW} and \ref{prmostgeneraldis} we obtained the following result.
\begin{proposition}\label{Prop:1.16}\cite[Proposition 4.3, Corollaries 4.4 and 4.5]{Part I}
Let  $\mathbf A$, $\mathbf B$, $\mathbf C$, and $\mathcal T$ be as in Proposition \ref{prmostgeneraldis}. Assume that  $k_{\mathbf C}\geq 1$ and $m=m_{\mathbf C}:=R-r_{\mathbf C}+2$. Then
\begin{equation}\label{manyimplicationsm}
\begin{aligned}
&\eqref{Kruskal}\ \xRightarrow{trivial}
\begin{tikzcd}[column sep=small,row sep=small]
& (\textup{C{\scriptsize m}}) \arrow[Rightarrow]{dr}{\eqref{maindiagramintro}} & \\
 (\textup{K{\scriptsize m}}) \arrow[Rightarrow]{ur}{\eqref{maindiagramintro}}\arrow[Rightarrow]{dr}{\eqref{maindiagramintro}} & & (\textup{U{\scriptsize m}})\\
 & (\textup{H{\scriptsize m}}) \arrow[Rightarrow]{ur}{\eqref{maindiagramintro}} &
\end{tikzcd}\
\xRightarrow{\eqref{maindiagramintro}} \
\begin{cases}
(\textup{C{\scriptsize 1}})\\
\min(k_{\mathbf A},k_{\mathbf B})\geq m-1,\\
(\textup{W{\scriptsize m}})
\end{cases}
\
\\
&\xRightarrow{\eqref{eq1.14}} \
\begin{cases}
(\textup{C{\scriptsize 1}})\\
(\textup{W{\scriptsize 1}}),\dots,(\textup{W{\scriptsize m}})
\end{cases}
\Rightarrow
\begin{cases}
r_{\mathcal T}=R,\\
\text{the third factor matrix of}\ \mathcal T\ \text{ is unique}.
\end{cases}
\end{aligned}
\end{equation}
\end{proposition}
Note that for  $r_{\mathbf C}=R$, we have $m=2$ and  (\textup{U{\scriptsize 2}}) is equivalent to (\textup{W{\scriptsize 2}}).
Moreover, in this case (\textup{U{\scriptsize 2}}) is necessary for uniqueness. We obtain the following counterpart of
Proposition \ref{Prop:1.16}.
\begin{proposition}\label{K2C2U2unique}\cite{DeLathauwer2006,JiangSid2004,Stegeman2009}
Let  $\mathbf A$, $\mathbf B$, $\mathbf C$, and $\mathcal T$ be as in Proposition \ref{prmostgeneraldis}. Assume that  $r_{\mathbf C}=R$. Then
\begin{equation}\label{eq:1.11}
\eqref{Kruskal}\ \Rightarrow
\begin{tikzcd}[column sep=tiny,row sep=tiny]
& (\textup{C{\scriptsize 2}}) \arrow[Rightarrow]{dr} & \\
 (\textup{K{\scriptsize 2}}) \arrow[Rightarrow]{ur}\arrow[Rightarrow]{dr} & & (\textup{U{\scriptsize 2}})\\
 & (\textup{H{\scriptsize 2}}) \arrow[Rightarrow]{ur} &
\end{tikzcd}\
\Leftrightarrow\
\begin{cases}
r_{\mathcal T}=R,\\
\text{the CPD of}\ \mathcal T\ \text{ is unique}.
\end{cases}
\end{equation}
\end{proposition}
\subsubsection{Generic conditions}
Let the matrices $\mathbf A\in\mathbb F^{I\times R}$, $\mathbf B\in\mathbb F^{J\times R}$ and $\mathbf C\in\mathbb F^{K\times R}$ be
randomly sampled from a continuous distribution.
It can be easily checked that the equations
$$
k_{\mathbf A}=r_{\mathbf A}=\min(I,R),\quad k_{\mathbf B}=r_{\mathbf B}=\min(J,R),\quad k_{\mathbf C}=r_{\mathbf C}=\min(K,R)
$$
hold generically.
Thus, by \eqref{Kruskal}, the CPD  of an $I\times J\times K$ tensor of rank $R$ is generically unique if
\begin{equation}\label{kruskalgeneric}
\min(I,R)+\min(J,R)+\min(K,R)\geq 2R+2.
\end{equation}
The generic uniqueness of one factor matrix has not yet been studied as such. It can be easily seen that in \eqref{manyimplicationsm} the  generic version of (\textup{K{\scriptsize m}}) for $m=R-K+2$ is also given by \eqref{kruskalgeneric}.

Let us additionally assume that $K\geq R$. Under this assumption, \eqref{kruskalgeneric}
reduces to
\begin{equation*}
\min(I,R)+\min(J,R)\geq R+2.
\end{equation*}
The generic version of condition $(\textup{C{\scriptsize 2}})$ was given in \cite{DeLathauwer2006, Psycho2006}.
It was indicated that the $C^2_IC^2_J\times C^2_R$ matrix $\mathbf U=\mathcal C_{2}(\mathbf A)\odot \mathcal C_2(\mathbf B)$ generically has full column rank
whenever the number of columns of $\mathbf U$ does not exceed the number of rows. By Proposition \ref{K2C2U2unique} the CPD  of an $I\times J\times K$ tensor of rank $R$ is then generically unique if
\begin{equation} \label{C2generic}
K\geq R\qquad\textup{ and }\  I(I-1)J(J-1)/4=C^2_IC^2_J\geq C^2_R=R(R-1)/2.
\end{equation}
The four following results have been obtained in algebraic geometry.
\begin{theorem}\label{th:gen0}\cite[Corollary 3.7]{Strassen1983}
 Let $3\leq I\leq J\leq K$, $K-1\leq (I-1)(J-1)$, and let $K$ be odd. Then the CPD
of an $I\times J\times K$ tensor of rank $R$ is generically unique if $R\leq IJK/(I+J+K-2)-K$.
\end{theorem}

\begin{theorem}\label{th:gen1} \cite[Theorem 1.1]{ChiantiniandOttaviani}
 Let $I\leq J\leq K$. Let $\alpha$, $\beta$ be maximal such that $2^\alpha\leq I$  and $2^\beta\leq J$. Then the CPD
of an $I\times J\times K$ tensor of rank $R$ is generically unique if $R\leq 2^{\alpha+\beta-2}$.
\end{theorem}
\begin{theorem}\label{th:fcr} \cite[Proposition 5.2]{ChiantiniandOttaviani},\cite[Theorem 2.7]{Strassen1983}
Let $R\leq (I-1)(J-1)\leq K$. Then the CPD of an $I\times J\times K$  tensor of rank $R$ is generically unique.
\end{theorem}

\begin{theorem}\cite[Theorem 1.2]{ChiantiniandOttaviani}\label{cubic}
The CPD of an $I\times I\times I$  tensor of rank $R$ is generically unique if $R\leq k(I)$, where
$k(I)$ is given in Table \ref{tab1}.
\begin{table}[htbp]
\caption{Upper bound $k(I)$ on $R$ under which generic uniqueness of the CPD of a $I\times I\times I$ tensor is guaranteed by Theorem \ref{cubic}.}
\begin{center}\footnotesize
\begin{tabular}{|c|ccccccccc|}
  \hline
  $I$   & 2& 3& 4& 5&  6&  7&  8&  9& 10\\
  \hline
  $k(I)$& 2& 3& 5& 9& 13& 18& 22& 27& 32\\
  \hline
 \end{tabular}
\end{center}
\label{tab1}
\end{table}
\end{theorem}

Finally, for a number of specific cases of dimensions and rank, generic uniqueness results have been obtained in \cite{tenBerge2004}.

\subsection{Results and organization}\label{Subsection1.4}
In this paper we use the conditions in \eqref{manyimplicationsm}  to establish CPD uniqueness in cases where $r_{\mathbf C} < R$.

In \S \ref{Section5} we assume that a tensor admits two PDs that have one or two factor matrices in common. We establish conditions under which both decompositions are the same. We obtain the following results.
\begin{proposition}\label{proponematrixisunique}
Let $\mathcal T=[\mathbf A, \mathbf B, \mathbf C]_R=[\bar{\mathbf A}, \bar{\mathbf B}, \mathbf C\mathbf{\Pi}{\mathbf \Lambda}_{\mathbf C}]_R$, where
$\mathbf{\Pi}$ is an $R\times R$ permutation matrix and ${\mathbf \Lambda}_{\mathbf C}$ is a nonsingular diagonal matrix. Let  the matrices $\mathbf A$, $\mathbf B$ and $\mathbf C$
satisfy the following condition
\begin{equation}\label{eq5.1}
\max(\min( k_{\mathbf A}, k_{\mathbf B}-1),\ \min( k_{\mathbf A}-1, k_{\mathbf B}))+k_{\mathbf C}\geq R+1.
\end{equation}
Then there exist nonsingular diagonal matrices ${\mathbf \Lambda}_{\mathbf A}$ and ${\mathbf \Lambda}_{\mathbf B}$ such that
$$\bar{\mathbf A}=\mathbf A\mathbf{\Pi}{\mathbf \Lambda}_{\mathbf A},\qquad
\bar{\mathbf B}=\mathbf B\mathbf{\Pi}{\mathbf \Lambda}_{\mathbf B},\qquad {\mathbf \Lambda}_{\mathbf A}{\mathbf \Lambda}_{\mathbf B}{\mathbf \Lambda}_{\mathbf C}=\mathbf I_R.
$$
\end{proposition}
\begin{proposition}\label{proptwomatrixisunique}
Let $\mathcal T=[\mathbf A, \mathbf B, \mathbf C]_R=[\mathbf A\mathbf{\Pi}_{\mathbf A}{\mathbf \Lambda}_{\mathbf A}, \bar{\mathbf B}, \mathbf C\mathbf{\Pi}_{\mathbf C}{\mathbf \Lambda}_{\mathbf C}]_R$, where
$\mathbf{\Pi}_{\mathbf A}$ and $\mathbf{\Pi}_{\mathbf C}$ are  $R\times R$ permutation matrices and where ${\mathbf \Lambda}_{\mathbf A}$ and ${\mathbf \Lambda}_{\mathbf C}$ are nonsingular diagonal matrices.
Let  the matrices $\mathbf A$, $\mathbf B$ and $\mathbf C$
satisfy  at least one of the following conditions
\begin{equation}\label{eq5.two}
\begin{split}
k_{\mathbf C}\geq 2\quad \text{ and }\quad \max(\min( k_{\mathbf A}, k_{\mathbf B}-1),\ \min( k_{\mathbf A}-1, k_{\mathbf B}))+r_{\mathbf C}\geq R+1,\\
k_{\mathbf A}\geq 2\quad \text{ and }\quad \max(\min( k_{\mathbf B}, k_{\mathbf C}-1),\ \min( k_{\mathbf B}-1, k_{\mathbf C}))+r_{\mathbf A}\geq R+1.
\end{split}
\end{equation}
Then $\mathbf{\Pi}_{\mathbf A}=\mathbf{\Pi}_{\mathbf C}$ and  $\bar{\mathbf B}=\mathbf B\mathbf{\Pi}_{\mathbf A}{\mathbf \Lambda}_{\mathbf A}^{-1}{\mathbf \Lambda}_{\mathbf C}^{-1}$.
\end{proposition}

Note that in Propositions \ref{proponematrixisunique} and \ref{proptwomatrixisunique} we do not assume that $R$ is minimal. Neither do we assume in Proposition \ref{proptwomatrixisunique} that $\mathbf{\Pi}_{\mathbf A}$ and $\mathbf{\Pi}_{\mathbf C}$ are the same.

In \S \ref{Section uniqueness CPD} we obtain new results concerning the uniqueness of the overall CPD
by combining  \eqref{manyimplicationsm}  with results from \S \ref{Section5}.

Combining \eqref{manyimplicationsm}  with Proposition \ref{proponematrixisunique}  we prove the following  statements.
\begin{proposition}\label{ProFullUniq1onematrixWm}
Let $\mathcal T=[\mathbf A, \mathbf B, \mathbf C]_R$ and
$m_{\mathbf C}:=R-r_{\mathbf C}+2$.
Assume that
\begin{romannum}
\item
condition \eqref{eq5.1} holds;
\item
condition $\text{\textup{(W{\scriptsize m}}}_{\mathbf C}$$\text{\textup{)}}$ holds for $\mathbf A$, $\mathbf B$, and $\mathbf C$;
\item
$\mathbf A\odot\mathbf B$ has full column rank.\hfill $(\textup{C{\scriptsize 1}})$
\end{romannum}
Then $r_{\mathcal T}=R$ and the CPD of tensor $\mathcal T$  is unique.
\end{proposition}
\begin{corollary}\label{ProFullUniq1onematrix}
Let $\mathcal T=[\mathbf A, \mathbf B, \mathbf C]_R$ and
$m_{\mathbf C}:=R-r_{\mathbf C}+2$.
Assume that
\begin{romannum}
\item
condition \eqref{eq5.1} holds;
\item
condition $\text{\textup{(U{\scriptsize m}}}_{\mathbf C}$$\text{\textup{)}}$ holds for $\mathbf A$ and $\mathbf B$.
\end{romannum}
Then $r_{\mathcal T}=R$ and the CPD of tensor $\mathcal T$  is unique.
\end{corollary}

\begin{corollary}\label{ProFullUniq1onematrixHm}
Let $\mathcal T=[\mathbf A, \mathbf B, \mathbf C]_R$ and
$m_{\mathbf C}:=R-r_{\mathbf C}+2$.
Assume that
\begin{romannum}
\item
condition \eqref{eq5.1} holds;
\item
condition $\text{\textup{(H{\scriptsize m}}}_{\mathbf C}$$\text{\textup{)}}$ holds for $\mathbf A$ and $\mathbf B$.
\end{romannum}
Then $r_{\mathcal T}=R$ and the CPD of tensor $\mathcal T$  is unique.
\end{corollary}

\begin{corollary}\label{ProFullUniq1onematrixcor1}
Let $\mathcal T=[\mathbf A, \mathbf B, \mathbf C]_R$ and
$m_{\mathbf C}:=R-r_{\mathbf C}+2$.
Assume that
\begin{romannum}
\item
condition \eqref{eq5.1} holds;
\item
$\mathcal C_{m_{\mathbf C}}(\mathbf A)\odot \mathcal C_{m_{\mathbf C}}(\mathbf B)$ has full column rank.
\end{romannum}
Then $r_{\mathcal T}=R$ and the CPD of  tensor $\mathcal T$  is unique.
\end{corollary}

Note that Proposition \ref{K2C2U2unique} is a special case of the results in  Proposition \ref{ProFullUniq1onematrixWm},
Corollaries \ref{ProFullUniq1onematrix}--\ref{ProFullUniq1onematrixcor1} and Kruskal's Theorem \ref{theoremKruskal}.
In the former, one factor matrix is assumed to have full column rank $(r_{\mathbf C}=R)$ while in the latter this is not necessary
($r_{\mathbf C}=R-m_{\mathbf C}+2$ with $m_{\mathbf C}\geq 2$). The condition on $\mathbf C$ is relaxed by tightening the conditions on $\mathbf A$ and $\mathbf B$.
For instance, Corollary \ref{ProFullUniq1onematrix} allows  $r_{\mathbf C}=R-m_{\mathbf C}+2$ with $m:=m_{\mathbf C}\geq 2$ by imposing \eqref{eq5.1}
and $(\textup{C{\scriptsize m}})$.
From scheme \eqref{maindiagramintro} we have that $(\textup{C{\scriptsize m}})$
implies $(\textup{C{\scriptsize 2}})$, and hence
$(\textup{C{\scriptsize m}})$ is more restrictive than $(\textup{C{\scriptsize 2}})$.
Scheme \eqref{maindiagramintro} further shows that Corollary \ref{ProFullUniq1onematrix}
is more general than Corollaries \ref{ProFullUniq1onematrixHm} and \ref{ProFullUniq1onematrixcor1}.
In turn, Proposition \ref{ProFullUniq1onematrixWm} is more general than Corollary \ref{ProFullUniq1onematrix}.
Note that we did not formulate a combination of implication $(\textup{K{\scriptsize m}})\Rightarrow (\textup{C{\scriptsize m}})$ (or (\textup{H{\scriptsize m}})) from scheme \eqref{manyimplicationsm}  with Proposition \ref{proponematrixisunique}. Such a  combination  leads to a result that is equivalent
to  Corollary \ref{Proptwokkk} below.

Combining \eqref{manyimplicationsm} with  Proposition \ref{proptwomatrixisunique} we prove the following results.
\begin{proposition}\label{ProFullUniq1}
Let $\mathcal T=[\mathbf A, \mathbf B, \mathbf C]_R$ and let
\begin{equation}
m_{\mathbf A}:=R-r_{\mathbf A}+2, \quad m_{\mathbf B}:=R-r_{\mathbf B}+2, \quad m_{\mathbf C}:=R-r_{\mathbf C}+2. \label{eqm_ABC}
\end{equation}
Assume that at least two of the following conditions hold
\begin{romannum}
\item
 condition $\text{\textup{(U{\scriptsize m}}}_{\mathbf A}$$\text{\textup{)}}$ holds for $\mathbf B$ and $\mathbf C$;
\item
 condition $\text{\textup{(U{\scriptsize m}}}_{\mathbf B}$$\text{\textup{)}}$ holds for $\mathbf C$ and $\mathbf A$;
\item
 condition $\text{\textup{(U{\scriptsize m}}}_{\mathbf C}$$\text{\textup{)}}$ holds for $\mathbf A$ and $\mathbf B$.
\end{romannum}
Then $r_{\mathcal T}=R$ and the CPD of  tensor $\mathcal T$  is unique.
\end{proposition}

\begin{corollary}\label{Prop5.4HmHm}
Let $\mathcal T=[\mathbf A, \mathbf B, \mathbf C]_R$ and consider
$m_{\mathbf A}$,  $m_{\mathbf B}$,  and  $m_{\mathbf C}$
defined in \eqref{eqm_ABC}.
Assume that at least two of the following conditions hold
\begin{romannum}
\item
 condition $\text{\textup{(H{\scriptsize m}}}_{\mathbf A}$$\text{\textup{)}}$ holds for $\mathbf B$ and $\mathbf C$;
\item
 condition $\text{\textup{(H{\scriptsize m}}}_{\mathbf B}$$\text{\textup{)}}$ holds for $\mathbf C$ and $\mathbf A$;
\item
 condition $\text{\textup{(H{\scriptsize m}}}_{\mathbf C}$$\text{\textup{)}}$ holds for $\mathbf A$ and $\mathbf B$.
\end{romannum}
Then $r_{\mathcal T}=R$ and the CPD of  tensor $\mathcal T$  is unique.
\end{corollary}

\begin{corollary}\label{Prop5.4}
Let $\mathcal T=[\mathbf A, \mathbf B, \mathbf C]_R$ and consider
$m_{\mathbf A}$,  $m_{\mathbf B}$,  and  $m_{\mathbf C}$
defined in \eqref{eqm_ABC}.
Let at least two of the  matrices
\begin{equation}
\mathcal C_{m_{\mathbf A}}(\mathbf B)\odot \mathcal C_{m_{\mathbf A}}(\mathbf C), \quad
\mathcal C_{m_{\mathbf B}}(\mathbf C)\odot \mathcal C_{m_{\mathbf B}}(\mathbf A), \quad
\mathcal C_{m_{\mathbf C}}(\mathbf A)\odot \mathcal C_{m_{\mathbf C}}(\mathbf B)\label{cond13C}
\end{equation}
have full column rank.
Then $r_{\mathcal T}=R$ and the CPD of  tensor $\mathcal T$  is unique.
\end{corollary}
\begin{corollary}\label{Proptwokkk}
Let $\mathcal T=[\mathbf A, \mathbf B, \mathbf C]_R$ and
let
 $(\mathbf X,\mathbf Y,\mathbf Z)$ coincide with
 $(\mathbf A,\mathbf B,\mathbf C)$, $(\mathbf B,\mathbf C,\mathbf A)$, or $(\mathbf C,\mathbf A,\mathbf B)$. If
 \begin{equation}\label{eqtwomatrK2}
 \left\{
\begin{array}{ll}
k_{\mathbf X}+r_{\mathbf Y}+r_{\mathbf Z}                                         &\geq 2R+2,\\
\min (r_{\mathbf Z}+k_{\mathbf Y},k_{\mathbf Z}+r_{\mathbf Y})    &                \geq\ R+2,
\end{array}
\right.
 \end{equation}
then $r_{\mathcal T}=R$ and the CPD of  tensor $\mathcal T$  is unique.
\end{corollary}
\begin{corollary}\label{corr:1.30}
Let $\mathcal T=[\mathbf A, \mathbf B, \mathbf C]_R$ and let  the following conditions hold
\begin{equation}\label{kruskalrelaxed}
\begin{cases}
k_{\mathbf A}   + r_{\mathbf B}+ r_{\mathbf C} \geq\  2R+2,\\
r_{\mathbf A}   + k_{\mathbf B}+ r_{\mathbf C} \geq\  2R+2,\\
r_{\mathbf A}   + r_{\mathbf B}+ k_{\mathbf C} \geq\  2R+2.
\end{cases}
\end{equation}
Then $r_{\mathcal T}=R$ and the CPD of  tensor $\mathcal T$  is unique.
\end{corollary}

Let us compare Kruskal's Theorems \ref{theoremKruskal}--\ref{theoremKruskalnew2} with Corollaries \ref{ProFullUniq1onematrixHm}, \ref{Prop5.4HmHm}, \ref{Proptwokkk}, and \ref{corr:1.30}.
Elementary algebra yields that  Theorem \ref{theoremKruskalnew1}
is equivalent to Corollary \ref{Proptwokkk}. From Corollary \ref{Prop5.4HmHm} it follows that assumption \textup{(i)} of Theorem
\ref{theoremKruskalnew2} is redundant.
  We will demonstrate in Examples \ref{example5.7} and \ref{OneHmbettertwoHm}  that
  it is not possible to state in general which of  the Corollaries \ref{ProFullUniq1onematrixHm} or \ref{Prop5.4HmHm} is more relaxed.
  Thus, Corollary \ref{ProFullUniq1onematrixHm} (obtained by  combining  implication $(\textup{H{\scriptsize m}})\Rightarrow (\textup{U{\scriptsize m}})$  from scheme \eqref{manyimplicationsm}  with Proposition \ref{proptwomatrixisunique})  is an
$\text{\textup{(H{\scriptsize m})}}$--type result on uniqueness that was not in
\cite{Kruskal1977}.
 Corollary \ref{corr:1.30} is a special case of Corollary \ref{Proptwokkk}, which is obviously more relaxed than Kruskal's well-known  Theorem \ref{theoremKruskal}.
 Finally we note that if condition $\text{\textup{(H{\scriptsize m})}}$ holds, then $r_{\mathbf A}+r_{\mathbf B}+r_{\mathbf C}\geq2R+2$. Thus, neither  Kruskal's Theorems \ref{theoremKruskal}--\ref{theoremKruskalnew2} nor Corollaries \ref{ProFullUniq1onematrixHm}, \ref{Prop5.4HmHm}, \ref{Proptwokkk}, \ref{corr:1.30} can be
 used for demonstrating  the uniqueness of a PD $[\mathbf A,\mathbf B,\mathbf C]_R$ when $r_{\mathbf A}+r_{\mathbf B}+r_{\mathbf C}<2R+2$.

We did not present a result based on a  combination of  $(\textup{W{\scriptsize m}})$-type implications from  scheme \eqref{manyimplicationsm}  with Proposition \ref{proptwomatrixisunique} because we do not have examples of cases where such conditions are more relaxed than those in
Proposition \ref{ProFullUniq1}.

In \S \ref{sect:symm} we indicate how our results can be adapted in the case of PD symmetries.

Well-known necessary conditions for the uniqueness of the CPD are \cite[p. 2079, Theorem 2]{LiuSid2001}, \cite[p. 28]{Krijnen1991},\cite[p. 651]{Strassen 1983}
\begin{align}
&\min(k_{\mathbf A},k_{\mathbf B},k_{\mathbf C})\geq 2,\label{nec1}\\
&\mathbf A\odot\mathbf B,\quad \mathbf B\odot\mathbf C,\quad \mathbf C\odot\mathbf A\quad\text{  have full column rank}.\label{nec2}
\end{align}
Further, the following necessary condition was obtained in \cite[Theorem 2.3]{LievenLL1}
\begin{equation}\label{nec3}
\text{ (\textup{U{\scriptsize 2}})  holds for  pairs }\ \ (\mathbf A,\mathbf B),\ \ (\mathbf B,\mathbf C),\  \text{ and }\  (\mathbf C,\mathbf A).
\end{equation}
It follows from scheme \eqref{maindiagramintro}
that \eqref{nec3} is more restrictive than \eqref{nec1} and \eqref{nec2}.
Our most general condition concerning uniqueness of  one factor matrix  is given in Proposition \ref{prmostgeneraldis}.
Note that in Proposition \ref{prmostgeneraldis}, condition (i) is more relaxed than  \eqref{nec1} and condition (iii) coincides with \eqref{nec2}.
One may wonder whether condition (iv) in Proposition \ref{prmostgeneraldis} is necessary for the uniqueness of at least one factor matrix.
In \S \ref{Appendix Extra} we show that this is not the case. We actually study an example in which CPD uniqueness can be established without $\text{\textup{(W{\scriptsize m})}}$ being satisfied.

In \S \ref{Section6} we study generic uniqueness of one factor matrix and generic CPD uniqueness. Our result on overall CPD uniqueness is the following.
\begin{proposition}\label{prop111overallbig}
The  CPD of  an $I\times J\times K$ tensor of rank $R$ is
generically unique if there exist matrices $\mathbf A_0\in \mathbb F^{I\times R}$,  $\mathbf B_0\in \mathbb F^{J\times R}$, and
$\mathbf C_0\in \mathbb F^{K\times R}$ such that at least one of the following conditions holds:
\begin{itemize}
\item[\textup{(i)}]
$\mathcal C_{m_{\mathbf C}}(\mathbf A_0)\odot \mathcal C_{m_{\mathbf C}}(\mathbf B_0)$  has full column rank, where $m_{\mathbf C}=R-\min (K,R)+2$;
\item[\textup{(ii)}]
$\mathcal C_{m_{\mathbf A}}(\mathbf B_0)\odot \mathcal C_{m_{\mathbf A}}(\mathbf C_0)$  has full column rank, where $m_{\mathbf A}=R-\min (I,R)+2$;
\item[\textup{(iii)}]
$\mathcal C_{m_{\mathbf B}}(\mathbf C_0)\odot \mathcal C_{m_{\mathbf B}}(\mathbf A_0)$  has full column rank, where $m_{\mathbf B}=R-\min (J,R)+2$.
\end{itemize}
\end{proposition}
We give several examples that illustrate the uniqueness results in the generic case.


\section{Equality of PDs with common factor matrices}\label{Section5}

In this section we assume that a tensor admits two not necessarily canonical PDs that have one or two factor matrices in common. In the latter case, the two PDs may have the columns of the common factor matrices permuted differently. We establish conditions that guarantee that the two PDs are the same.
\subsection{One factor matrix in common} \label{subsection5.1}

In this subsection we assume that two PDs have the factor matrix $\mathbf C$ in common. The result that we are concerned with, is Proposition \ref{proponematrixisunique}. The proof is based on the following three lemmas.

\begin{lemma}\label{lemma0.1unique}
For matrices $\mathbf A, \bar{\mathbf A}\in\mathbb F^{I\times R}$ and indices
$r_1,\dots, r_n \in\{1,\dots, R\}$ define the subspaces
$E_{r_1\dots r_n}$ and $\bar{E}_{r_1\dots r_n}$ as follows
\begin{equation*}
E_{r_1\dots r_n}:=\textup{span}\{\mathbf a_{r_1},\dots, \mathbf  a_{r_n}\},\qquad
\bar{E}_{r_1\dots r_n}:=\textup{span}\{\bar{\mathbf a}_{r_1},\dots, \bar{\mathbf  a}_{r_n}\}.
\end{equation*}
Assume that $k_{\mathbf A} \geq 2$ and that there exists  $m\in\{2,\dots, k_{\mathbf A}\}$ such that
\begin{equation}
E_{r_1\dots r_{m-1}}\subseteq \bar{E}_{r_1\dots r_{m-1}}\qquad\text{ for all }\qquad  1\leq r_1<r_2<\dots<r_{m-1}\leq R.
\label{equnique1}
\end{equation}
Then there exists a nonsingular  diagonal matrix ${\mathbf \Lambda}$ such that $\mathbf A=\bar{\mathbf A}{\mathbf \Lambda}$.
\end{lemma}

\begin{proof}
For $m=2$ we have
\begin{equation}
\textup{span}\{\mathbf a_{r_1}\}=E_{r_1}\subseteq\bar{E}_{r_1}=\textup{span}\{\bar{\mathbf a}_{r_1}\},\qquad \textup{  for all  } 1\leq r_1\leq R,
\label{spanproof1}
\end{equation}
such that the Lemma trivially holds. For $m\geq 3$ we arrive at (\ref{spanproof1}) by downward induction on  $l=m,m-1,\dots,3$. Assuming that
\begin{equation}
E_{r_1\dots r_{l-1}}\subseteq \bar{E}_{r_1\dots r_{l-1}}\qquad\text{ for all }\qquad  1\leq r_1<r_2<\dots<r_{l-1}\leq R,\label{eqesubsete}
\end{equation}
we show that
\begin{equation*}
E_{r_1 \dots r_{l-2}}\subseteq \bar{E}_{r_1 \dots r_{l-2}}\qquad\text{ for all }\qquad  1\leq r_1<r_2<\dots<r_{l-2}\leq R.
\end{equation*}
Assume $r_1, r_2, \dots, r_{l-2}$ fixed and let $i,j\in\{1,\dots,R\}\setminus \{r_1,\dots,r_{l-2}\}$, with $i\ne j$.
Since $l \leq m \leq k_{\mathbf A}$, we have that $\dim E_{r_1,\dots,r_{l-2},i,j}=l$.
Because
\begin{equation*}
\begin{split}
l=\dim E_{r_1,\dots,r_{l-2},i,j}&\ \leq \dim\textup{span}\{E_{r_1,\dots,r_{l-2},i},\ E_{r_1,\dots,r_{l-2},j}\}  \\
&\stackrel{\eqref{eqesubsete}}{\leq}\dim
\textup{span}\{\bar{E}_{r_1,\dots,r_{l-2},i},\ \bar{E}_{r_1,\dots,r_{l-2},j}\}
\end{split}
\end{equation*}
we have
\begin{equation}
\bar{E}_{r_1,\dots,r_{l-2},i}\ne \bar{E}_{r_1,\dots,r_{l-2},j}.\label{eqenoteqe}
\end{equation}
Therefore,
\begin{equation*}
\begin{split}
E_{r_1,\dots,r_{l-2}}&\ \subseteq \left(E_{r_1,\dots,r_{l-2},i}\cap E_{r_1,\dots,r_{l-2},j}\right)\\
&\stackrel{\eqref{eqesubsete}}{\subseteq}\left(\bar{E}_{r_1,\dots,r_{l-2},i}\cap \bar{E}_{r_1,\dots,r_{l-2},j}\right)\stackrel{\eqref{eqenoteqe}}{=}\bar{E}_{r_1,\dots,r_{l-2}}.
\end{split}
\end{equation*}
The induction follows. To conclude the proof, we note that ${\mathbf \Lambda}$ is nonsingular since $k_{\mathbf A} \geq 2$.
\qquad\end{proof}

\begin{lemma}\label{lemma0.2unique}
Let $\mathbf C\in\mathbb F^{K\times R}$ and consider $m$ such that  $m\leq k_{\mathbf C}$.
Then for any set of distinct indices $\mathcal I=\{i_1,\dots, i_{m-1}\}\subseteq\{1,\dots,R\}$ there exists a vector $\mathbf x\in\mathbb F^K$ such that
\begin{equation}
\mathbf x^T\mathbf c_i=0 \ \text{ for }\  i\in \mathcal I\ \text{and }\ \mathbf x^T\mathbf c_i\ne 0 \text{ for } i\in \mathcal I^{\textup{c}}:=\{1,\dots,R\}\setminus \mathcal I.
\label{equnique3}
\end{equation}
\end{lemma}
\begin{proof}
Let ${\mathbf C}_{\mathcal I} \in \mathbb F^{K \times (m-1)}$ and ${\mathbf C}_{{\mathcal I}^{\textup{c}}} \in\mathbb F^{K \times (R-m+1)}$ contain the columns of ${\mathbf C}$ indexed by ${\mathcal I}$ and ${\mathcal I}^{\textup{c}}$, respectively, and let the columns of ${\mathbf C}_{\mathcal I}^\perp \in \mathbb{F}^{K \times (K-m+1)}$ form a basis for the orthogonal complement of  $\mbox{range}({\mathbf C}_{\mathcal I})$. The matrix $({\mathbf C}_{\mathcal I}^\perp)^H {\mathbf C}_{{\mathcal I}^{\textup{c}}}$ cannot have a zero column, otherwise the corresponding column of ${\mathbf C}_{{\mathcal I}^{\textup{c}}}$ would be in $\mbox{range}({\mathbf C}_{\mathcal I})$, which would be a contradiction with $k_{\mathbf C}\geq m$. We conclude that (\ref{equnique3}) holds for
$\mathbf{x} = ({\mathbf C}_{\mathcal I}^\perp \mathbf{y})^*$, with $\mathbf{y} \in \mathbb{F}^{K-m+1}$ generic.
\qquad\end{proof}

\begin{lemma}\label{lemma:propP}
Let $\mathbf P$ be an $R\times R$ permutation matrix.
Then for any vector $\mathbf{\lambda}\in\mathbb F^R$,
\begin{equation}\label{eq:PropP}
\textup{Diag}(\mathbf{\Pi}{\mathbf \lambda})\mathbf{\Pi}=\mathbf{\Pi}\textup{Diag}({\mathbf \lambda}).
\end{equation}
\end{lemma}
\begin{proof}
The lemma follows directly from the definition of  permutation matrix.
\end{proof}

We are now ready to prove Proposition \ref{proponematrixisunique}.
\begin{proof}
Let  $\widehat{\mathbf A}:=\bar{\mathbf A}\mathbf{\Pi}^T$ and $\widehat{\mathbf B}:=\bar{\mathbf B}{\mathbf \Lambda}_{\mathbf C}^{-1}\mathbf{\Pi}^T$. Then
\begin{equation}\label{eqProofPr5.1}
\mathcal T=[\mathbf A, \mathbf B, \mathbf C]_R=
[\bar{\mathbf A}, \bar{\mathbf B}, \mathbf C\mathbf{\Pi}{\mathbf \Lambda}_{\mathbf C}]_R=
[\widehat{\mathbf A}, \widehat{\mathbf B}, \mathbf C]_R.
\end{equation}
We show
that the columns of $\mathbf A$ and $\mathbf B$ coincide up to scaling with the corresponding columns of
$\widehat{\mathbf A}$ and $\widehat{\mathbf B}$, respectively.
Consider indices $i_1$, \ldots, $i_{R-k_{\mathbf C}+1}$  such that $1\leq i_1<\dots<i_{R-k_{\mathbf C}+1}\leq R$. Let $m:=k_{\mathbf C}$ and let
$\mathcal I:=\{1,\dots,R\}\setminus\{i_1,\dots,i_{R-k_{\mathbf C}+1}\}$.
From Lemma \ref{lemma0.2unique} it follows that there exists a  vector $\mathbf x\in \mathbb F^K$ such that
$$
\mathbf x^T\mathbf c_i=0\ \text{ for }\  i\in \mathcal I\ \text{ and }
\mathbf x^T\mathbf c_i\ne 0\ \text{ for }\ i\in \mathcal I^{\textup{c}}=\{i_1,\dots,i_{R-k_{\mathbf C}+1}\}.
$$
Let $\mathbf d=\left[\begin{matrix} \mathbf x^T\mathbf c_{i_1}&\dots& \mathbf x^T\mathbf c_{i_{R-k_{\mathbf C}+1}} \end{matrix}\right]^T$. Then
 $(\mathbf A\odot \mathbf B)\mathbf C^T\mathbf x=(\widehat{\mathbf A}\odot \widehat{\mathbf B})\mathbf C^T\mathbf x$ is equivalent to
\begin{equation*}
\begin{split}
&\left(\left[\begin{matrix} \mathbf a_{i_1}&\dots& \mathbf a_{i_{R-k_{\mathbf C}+1}} \end{matrix}\right]\odot
\left[\begin{matrix} \mathbf b_{i_1}&\dots& \mathbf b_{i_{R-k_{\mathbf C}+1}} \end{matrix}\right]\right)\mathbf d
=\\
&\left(\left[\begin{matrix} \widehat{\mathbf a}_{i_1}&\dots& \widehat{\mathbf a}_{i_{R-k_{\mathbf C}+1}} \end{matrix}\right]\odot
\left[\begin{matrix} \widehat{\mathbf b}_{i_1}&\dots& \widehat{\mathbf b}_{i_{R-k_{\mathbf C}+1}} \end{matrix}\right]\right)\mathbf d,
\end{split}
\end{equation*}
which may be expressed as
\begin{align*}
\left[\begin{matrix} \mathbf a_{i_1}&\dots& \mathbf a_{i_{R-k_{\mathbf C}+1}} \end{matrix}\right]\textup{\text{Diag}}(\mathbf d)
&\left[\begin{matrix} \mathbf b_{i_1}&\dots& \mathbf b_{i_{R-k_{\mathbf C}+1}} \end{matrix}\right]^T
\\
=&\left[\begin{matrix} \widehat{\mathbf a}_{i_1}&\dots& \widehat{\mathbf a}_{i_{R-k_{\mathbf C}+1}} \end{matrix}\right]\textup{\text{Diag}}(\mathbf d)
\left[\begin{matrix} \widehat{\mathbf b}_{i_1}&\dots& \widehat{\mathbf b}_{i_{R-k_{\mathbf C}+1}} \end{matrix}\right]^T.
\end{align*}
By \eqref{eq5.1}, $\min (k_{\mathbf A},k_{\mathbf B})\geq R-k_{\mathbf C}+1$.
Hence, the matrices $\left[\begin{matrix} \mathbf a_{i_1}&\dots& \mathbf a_{i_{R-k_{\mathbf C}+1}} \end{matrix}\right]$
 and $\left[\begin{matrix} \mathbf b_{i_1}&\dots& \mathbf b_{i_{R-k_{\mathbf C}+1}} \end{matrix}\right]$ have full column rank.
Since by
construction the vector $\mathbf d$ has only nonzero components, it follows that
\begin{align*}
\mathbf a_{i_1},\dots,\mathbf a_{i_{R-k_{\mathbf C}+1}}&\in\textup{span}\{\widehat{\mathbf a}_{i_1},\dots,\widehat{\mathbf a}_{i_{R-k_{\mathbf C}+1}}\},\\
\mathbf b_{i_1},\dots,\mathbf b_{i_{R-k_{\mathbf C}+1}}&\in\textup{span}\{\widehat{\mathbf b}_{i_1},\dots,\widehat{\mathbf b}_{i_{R-k_{\mathbf C}+1}}\}.
\end{align*}
By \eqref{eq5.1}, $\max (k_{\mathbf A},k_{\mathbf B})\geq m:=R-k_{\mathbf C}+2 \geq 2$.
Without loss of generality we confine ourselves to the case  $k_{\mathbf A}\geq m$. Then,
by Lemma \ref{lemma0.1unique}, there exists a nonsingular diagonal matrix ${\mathbf \Lambda}$ such that
$\mathbf A=\widehat{\mathbf A}{\mathbf \Lambda}$.
Denoting $\lambda_{\mathbf A}:=\mathbf{\Pi}^T\textup{diag}(\mathbf{\Lambda}^{-1})$ and $\mathbf\Lambda_{\mathbf A}=\textup{Diag}(\mathbf{\lambda}_{\mathbf A})$
and applying Lemma \ref{lemma:propP}, we have
$$
\bar{\mathbf A}=\widehat{\mathbf A}\mathbf{\Pi}=\mathbf A{\mathbf \Lambda}^{-1}\mathbf{\Pi}=\mathbf A\textup{Diag}(\mathbf{\Pi}{\mathbf \lambda}_{\mathbf A})\mathbf{\Pi}=\mathbf A\mathbf{\Pi}\textup{Diag}({\mathbf \lambda}_{\mathbf A})
=\mathbf A\mathbf{\Pi}{\mathbf \Lambda}_{\mathbf A}.
$$
It follows from \eqref{eqProofPr5.1} and \eqref{eqT_V} that
$$
(\mathbf C\odot\mathbf A)\mathbf B^T=
(\mathbf C\mathbf{\Pi}{\mathbf \Lambda}_{\mathbf C}\odot\bar{\mathbf A})\bar{\mathbf B}^T=
(\mathbf C\mathbf{\Pi}{\mathbf \Lambda}_{\mathbf C}\odot\mathbf A\mathbf{\Pi}{\mathbf \Lambda}_{\mathbf A})\bar{\mathbf B}^T=
(\mathbf C\odot\mathbf A)\mathbf{\Pi}{\mathbf \Lambda}_{\mathbf C}{\mathbf \Lambda}_{\mathbf A}\bar{\mathbf B}^T.
$$
 Since $k_{\mathbf A}\geq R-k_{\mathbf C}+2$, it follows
that condition (\text{K{\scriptsize 1}}) holds for the matrices $\mathbf A$ and $\mathbf C$.
From Proposition \ref{Prop:KCUW} (1) it follows that the matrix $\mathbf C\odot\mathbf A$
has full column rank. Hence, $\mathbf B^T=\mathbf{\Pi}{\mathbf \Lambda}_{\mathbf C}{\mathbf \Lambda}_{\mathbf A}\bar{\mathbf B}^T$, i.e., $\bar{\mathbf B}=\mathbf B\mathbf{\Pi}{\mathbf \Lambda}_{\mathbf A}^{-1}{\mathbf \Lambda}_{\mathbf C}^{-1}=:\mathbf B\mathbf{\Pi}{\mathbf \Lambda}_{\mathbf B}$.
\qquad\end{proof}

\begin{example}\label{ex:2.4}
Consider the $2\times 3\times 3$ tensor given by
$\mathcal T=[\widehat{\mathbf A},\widehat{\mathbf B},\widehat{\mathbf C}]_3$, where
$$
\widehat{\mathbf A}=\left[\begin{array}{rrrr}1&1&1\\-1&-2&3\end{array}\right],\ \
\widehat{\mathbf B}=\left[\begin{array}{rrrr}6&12&2\\3&4&-1\\4& 6& -4\end{array}\right],\ \
\widehat{\mathbf C}=\left[\begin{array}{rrrr}1&0&0\\0&1&0\\0& 0& 1\end{array}\right].
$$
Since $k_{\widehat{\mathbf A}}+k_{\widehat{\mathbf B}}+k_{\widehat{\mathbf C}}=2+3+3\geq 2 \times 3+2$, it follows from Theorem
\ref{theoremKruskal} that $r_{\mathcal T}=3$ and that the CPD of $\mathcal T$ is unique.

Increasing the number of terms, we also have $\mathcal T=[\mathbf A,\mathbf B,\mathbf C]_4$ for
$$
\mathbf A=\left[\begin{array}{rrrr}1&0&1&1\\0&1&1&2\end{array}\right],\ \
\mathbf B=\left[\begin{array}{rrrr}1&1&0&0\\1&0&1&0\\1& 0& 0& 1\end{array}\right],\ \
\mathbf C=\left[\begin{array}{rrrr}6&-6&-3&-2\\12&-24&-8&-6\\2& 6& -3& -6\end{array}\right].
$$
Since  $k_{\mathbf A}=2$ and $k_{\mathbf B}=k_{\mathbf C}=3$, condition \eqref{eq5.1} holds.
Hence, by Proposition \ref{proponematrixisunique}, if $\mathcal T=[\bar{\mathbf A},\bar{\mathbf B},\bar{\mathbf C}]_4$ and $\bar{\mathbf C}=\mathbf C$,
then there exists a nonsingular diagonal matrix $\mathbf\Lambda$ such that $\bar{\mathbf A}=\mathbf A\mathbf\Lambda$ and $\bar{\mathbf B}=\mathbf B\mathbf\Lambda^{-1}$.

The following condition is also satisfied:
$$
\max(\min( k_{\mathbf A}, k_{\mathbf C}-1),\ \min( k_{\mathbf A}-1, k_{\mathbf C}))+k_{\mathbf B}\geq R+1.
$$
By symmetry, we have from Proposition \ref{proponematrixisunique} that, if $\mathcal T=[\bar{\mathbf A},\bar{\mathbf B},\bar{\mathbf C}]_4$ and $\bar{\mathbf B}=\mathbf B$,
then there exists a nonsingular diagonal matrix $\mathbf\Lambda$ such that $\bar{\mathbf A}=\mathbf A\mathbf\Lambda$ and $\bar{\mathbf C}=\mathbf C\mathbf\Lambda^{-1}$.

Finally, we show that the inequality of condition  \eqref{eq5.1} is sharp. We have
$$
\max(\min( k_{\mathbf B}, k_{\mathbf C}-1),\ \min( k_{\mathbf B}-1, k_{\mathbf C}))+k_{\mathbf A}=R< R+1.
$$
One can verify that  $\mathcal T=[\bar{\mathbf A},\bar{\mathbf B},\bar{\mathbf C}]_4$ with $\bar{\mathbf A}=\mathbf A$ and with $\bar{\mathbf B}$ and $\bar{\mathbf C}$ given by
 $$
 \bar{\mathbf B}=
 \left[
 \begin{array}{rrr}
  6& 12&   2\\
  3&  4&  -1\\
  4&  6&  -4
 \end{array}
\right]
  \left[
 \begin{array}{rrr}
 1&0&0\\
 0&\alpha&0\\
 0&0&\beta
\end{array}
\right]
\left[
 \begin{array}{rrrr}
1&   1&    1 &   1 \\
1&   2&  4/3&  3/2 \\
1&  -3&    3&    9 \\
\end{array}
\right],
 $$
 $$
 \bar{\mathbf C}=
 \left[
 \begin{array}{rrr}
 1&0&0\\
 0&1/\alpha&0\\
 0&0&1/\beta
\end{array}
\right]
  \left[
 \begin{array}{rrrr}
  6&      -6&      -3&      -2\\
 -24/5&  48/5&  16/5&  12/5\\
  2/15&   2/5&  -1/5&  -2/5\\
\end{array}
\right],
 $$
 for arbitrary nonzero $\alpha$ and $\beta$. Hence, there exist infinitely many PDs $\mathcal T=[\bar{\mathbf A},\bar{\mathbf B},\bar{\mathbf C}]_4$ with $\bar{\mathbf A}=\mathbf A$;
 the columns of $\bar{\mathbf B}$ and $\bar{\mathbf C}$ are only proportional to the columns
 of $\mathbf B$ and $\mathbf C$, respectively, for  $\alpha=-2/5$ and $\beta=1/15$.
We conclude that the inequality of condition  \eqref{eq5.1} is sharp.
\end{example}

\subsection{Two factor matrices in common}\label{subsection5.2}

In this subsection we assume that two PDs have the factor matrices $\mathbf A$ and  $\mathbf C$ in common. We do not assume however that in the two PDs the columns of these matrices are permuted in the same manner. The result that we are concerned with, is Proposition \ref{proptwomatrixisunique}.

\begin{proof}
Without loss of generality, we confine ourselves to the case
\begin{equation}
k_{\mathbf C}\geq 2\quad \text{ and } \quad \min( k_{\mathbf A}-1, k_{\mathbf B})+r_{\mathbf C}\geq R+1.
\label{eq:proofprop1.24intro}
\end{equation}
We set for brevity $r:=r_{\mathbf C}$. Denoting  $\mathbf{\Pi}=\mathbf{\Pi}_{\mathbf A}\mathbf{\Pi}_{\mathbf C}^T$ and
$\widehat{\mathbf B}=\bar{\mathbf B}{\mathbf \Lambda}_{\mathbf A}{\mathbf \Lambda}_{\mathbf C}\mathbf{\Pi}_{\mathbf C}^T$, we have
$[\mathbf A\mathbf{\Pi}_{\mathbf A}{\mathbf \Lambda}_{\mathbf A}, \bar{\mathbf B}, \mathbf C\mathbf{\Pi}_{\mathbf C}{\mathbf \Lambda}_{\mathbf C}]_R
= [\mathbf A\mathbf{\Pi}_{\mathbf A}\mathbf{\Pi}_{\mathbf C}^T, \bar{\mathbf B}{\mathbf \Lambda}_{\mathbf A}{\mathbf \Lambda}_{\mathbf C}\mathbf{\Pi}_{\mathbf C}^T, \mathbf C]_R=
[\mathbf A\mathbf{\Pi}, \widehat{\mathbf B}, \mathbf C]_R$.
We will show that, under \eqref{eq:proofprop1.24intro}, $[\mathbf A, \mathbf B, \mathbf C]_R = [\mathbf A\mathbf{\Pi}, \widehat{\mathbf B}, \mathbf C]_R$ implies that $\mathbf{\Pi}=\mathbf I_R$. This, in turn, immediately implies that
$\mathbf{\Pi}_{\mathbf A}=\mathbf{\Pi}_{\mathbf C}$ and  $\bar{\mathbf B}=\mathbf B\mathbf{\Pi}_{\mathbf A}{\mathbf \Lambda}_{\mathbf A}^{-1}{\mathbf \Lambda}_{\mathbf C}^{-1}$.

\begin{romannum}
\item Let us fix integers $i_1,\dots,i_r$ such that the columns $\mathbf c_{i_1},\dots,\mathbf c_{i_r}$ form a basis of $\textup{range}(\mathbf C)$
and let us set $\{j_1,\dots,j_{R-r}\}:=\{1,\dots,R\}\setminus\{i_1,\dots,i_r\}$.
Let $\mathbf X \in\mathbb F^{K\times r}$,
denote a  right inverse of $\left[\begin{matrix}\mathbf c_{i_1}&\dots&\mathbf c_{i_r}\end{matrix}\right]^T$,
i.e., $\left[\begin{matrix}\mathbf c_{i_1}&\dots&\mathbf c_{i_r}\end{matrix}\right]^T\mathbf X=\mathbf I_r$.
Define the subspaces $E, E_{i_k} \subseteq \mathbb F^R$ as follows:
 \begin{align*}
 E&=\textup{span}\{\mathbf e_{j_1}^R,\dots \mathbf e_{j_{R-r}}^R\},\\
 E_{i_k}&=\textup{span}\{\mathbf e_l^R:\ \mathbf c^T_l \mathbf x_k \ne 0,\ l\in\{j_1,\dots,j_{R-r} \} \},\qquad k\in\{1,\dots,r\}.
 \end{align*}
 By construction, $E_{i_k} \subseteq E$ and $\mathbf e_{i_l}^R \notin E_{i_k}$, $k,l \in \{1,\dots,r\}$.
\item  Let us show that
$\mathbf{\Pi} \textup{span} \{E_{i_k},\mathbf e_{i_k}^R\}=\textup{span}\{E_{i_k},\mathbf e_{i_k}^R\}$ for all $k\in\{1,\dots,r\}$.
Let us fix $k\in\{1,\dots,r\}$. Assume that  $\mathbf C^T\mathbf x_k$ has nonzero entries at positions $k_1,\dots,k_L$. Denote these entries by $\alpha_1, \dots, \alpha_L$. From the definition of $\mathbf X$ and $E_{i_k}$ it follows that $L \leq R-r+1$ and
$\textup{span}\{\mathbf e_{k_1}^R,\dots,\mathbf e_{k_L}^R\}=\textup{span}\{E_{i_k},\mathbf e_{i_k}^R\}$.

Define $\mathbf P_k=\left[\begin{matrix}\mathbf e_{k_1}^R&\dots&\mathbf e_{k_L}^R\end{matrix}\right]$. Then we have
\begin{eqnarray}
\mathbf P_k\mathbf P_k^T\textup{Diag}(\mathbf C^T\mathbf x_k)\mathbf P_k\mathbf P_k^T & = & \textup{Diag}(\mathbf C^T\mathbf x_k) ,\label{eq:proofprop1.24iia} \\
\mathbf P_k^T\textup{Diag}(\mathbf C^T\mathbf x_k)\mathbf P_k & = & \textup{Diag}( \left[\begin{matrix}\alpha_1&\dots&\alpha_L\end{matrix}\right]). \label{eqtwomatrices2}
\end{eqnarray}
Further, $[\mathbf A, \mathbf B, \mathbf C]_R = [\mathbf A\mathbf{\Pi}, \widehat{\mathbf B}, \mathbf C]_R$ implies that
\begin{equation}
\mathbf A \textup{Diag}(\mathbf C^T\mathbf x_k) \mathbf B^T=
\mathbf A\mathbf{\Pi} \textup{Diag}(\mathbf C^T \mathbf x_k) \widehat{\mathbf B}^T. \label{eq:proofprop1.24iib}
\end{equation}
Using \eqref{eq:proofprop1.24iia}--\eqref{eq:proofprop1.24iib}, we obtain
\begin{eqnarray}
\mathbf A\mathbf P_k\textup{Diag}( \left[\begin{matrix}\alpha_1&\dots&\alpha_L\end{matrix}\right])\mathbf P_k^T\mathbf B^T & = & \mathbf A \mathbf P_k\mathbf P_k^T\textup{Diag}(\mathbf C^T\mathbf x_k)\mathbf P_k\mathbf P_k^T \mathbf B^T \nonumber \\
 &= &\mathbf A \textup{Diag}(\mathbf C^T\mathbf x_k) \mathbf B^T \nonumber  \\
 & = &
\mathbf A\mathbf{\Pi} \textup{Diag}(\mathbf C^T\mathbf x_k) \widehat{\mathbf B}^T \nonumber  \\
& = & \mathbf A\mathbf{\Pi} \mathbf P_k\mathbf P_k^T\textup{Diag}(\mathbf C^T\mathbf x_k)\mathbf P_k\mathbf P_k^T \widehat{\mathbf B}^T \nonumber \\
& = &\mathbf A\mathbf{\Pi}\mathbf P_k\textup{Diag}( \left[\begin{matrix}\alpha_1&\dots&\alpha_L\end{matrix}\right])\mathbf P_k^T\widehat{\mathbf B}^T. \label{eq:proofprop1.24iic}
\end{eqnarray}
Note that $\mathbf B\mathbf P_k=\left[\begin{matrix}\mathbf b_{k_1}&\dots&\mathbf b_{k_L}\end{matrix}\right]$. Since, by \eqref{eq:proofprop1.24intro}, $k_{\mathbf B}\geq R-r+1\geq L$, it follows that the matrix $\mathbf P_k^T \widehat{\mathbf B}^T$ has full row rank. Further noting that $\mathbf A \mathbf P_k=\left[\begin{matrix}\mathbf a_{k_1}&\dots&\mathbf a_{k_L}\end{matrix}\right]$ and $\mathbf A \mathbf{\Pi} \mathbf P_k=\left[\begin{matrix} (\mathbf A \mathbf{\Pi})_{k_1}&\dots& (\mathbf A \mathbf{\Pi})_{k_L}\end{matrix}\right]$, we obtain from \eqref{eq:proofprop1.24iic} that
\begin{equation}\label{eqincl}
\textup{span}\{\mathbf a_{k_1},\dots,\mathbf a_{k_L}\}\subseteq \textup{span}\{(\mathbf A\mathbf{\Pi})_{k_1},\dots,(\mathbf A\mathbf{\Pi})_{k_L}\}.
\end{equation}
Since, by \eqref{eq:proofprop1.24intro}, $k_{\mathbf A}\geq R-r+2\geq L+1$,  \eqref{eqincl} is only possible if
$\mathbf{\Pi} \textup{span}\{E_{i_k},\mathbf e_{i_k}^R\}=\textup{span}\{E_{i_k},\mathbf e_{i_k}^R\}$.
\item Let us show that $\mathbf{\Pi}E=E$. Let us fix $j\in \{j_1,\dots,j_{R-r}\}$.
From $\mathbf X^T\mathbf c_{i_k}=\mathbf e_k^r$ for $k\in\{1,\dots,r\}$, the fact that the vectors
 $\mathbf c_{i_1},\dots,\mathbf c_{i_r}$ form a basis of $\textup{range}(\mathbf C)$, and $k_{\mathbf C}\geq 2$, it follows that
the vector $\mathbf X^T \mathbf c_j$ has at least two nonzero components, say, the $m$-th and $n$-th component. Since $\mathbf c^T_j \mathbf x_m \ne 0$ and $\mathbf c^T_j \mathbf x_n \ne 0$, we have $\mathbf e_j^R \in E_{i_m}\cap E_{i_n}$. From the preceding steps we have
\begin{align*}
\mathbf{\Pi}\mathbf e_j^R \in \mathbf{\Pi}(E_{i_m}\cap E_{i_n})
& \stackrel{\mbox{\scriptsize (i)}}{=}
\mathbf{\Pi}\left(\textup{span}\{E_{i_m},\mathbf e_{i_m}^R\}\cap \textup{span}\{E_{i_n},\mathbf e_{i_n}^R\}\right) \\
& \stackrel{\mbox{\scriptsize (ii)}}{\subseteq}
\textup{span}\{E_{i_m},\mathbf e_{i_m}^R\} \cap \textup{span}\{E_{i_n},\mathbf e_{i_n}^R\}
\stackrel{\mbox{\scriptsize (i)}}{=} E_{i_m} \cap E_{i_n} \subseteq E.
\end{align*}
Since this holds true for any index $j \in \{j_1,\dots,j_{R-r}\}$, it follows that $\mathbf{\Pi}E=E$.
\item
Let us show that
$\mathbf{\Pi} \mathbf e_{i_k}^R=\mathbf e_{i_k}^R$ for all $k\in\{1,\dots,r\}$. From the preceding steps we have
$$
\mathbf{\Pi} E_{i_k} \stackrel{\mbox{\scriptsize (i)}}{=} \mathbf{\Pi}\left(\textup{span}\{E_{i_k},\mathbf e_{i_k}^R\}\cap E\right)
\stackrel{\mbox{\scriptsize (ii), (iii)}}{\subseteq} \textup{span}\{E_{i_k},\mathbf e_{i_k}^R\}\cap E
\stackrel{\mbox{\scriptsize (i)}}{=}
E_{i_k}.
$$
On the other hand, we have from step (iii) that $\mathbf{\Pi}\textup{span}\{E_{i_k},\mathbf e_{i_k}^R\}=\{E_{i_k},\mathbf e_{i_k}^R\}$, with, as shown in step (i), $\mathbf e_{i_k}^R \notin E_{i_k}$. It follows that  $\mathbf{\Pi} \mathbf e_{i_k}^R=\mathbf e_{i_k}^R$ for all $k\in\{1,\dots,r\}$.
\item We have so far shown that, if
the columns $\mathbf c_{i_1},\dots,\mathbf c_{i_r}$ form a basis of $\textup{range}(\mathbf C)$, then
$\mathbf{\Pi}\left[\begin{matrix}\mathbf e_{i_1}^R&\dots \mathbf e_{i_r}^R\end{matrix} \right]$
$=\left[\begin{matrix}\mathbf e_{i_1}^R&\dots \mathbf e_{i_r}^R\end{matrix} \right]$. To complete the proof of the overall equality
$\mathbf{\Pi}=\mathbf I_R$, it suffices to note that a basis of $\textup{range}(\mathbf C)$ can be constructed starting from any column of $\mathbf C$.
 \end{romannum}
\end{proof}
\section{Overall CPD uniqueness}\label{Section uniqueness CPD}
In Proposition   \ref{ProFullUniq1onematrixWm} and Corollaries   \ref{ProFullUniq1onematrix}--\ref{ProFullUniq1onematrixcor1}
overall CPD uniqueness is derived from uniqueness of one factor matrix, where the latter is guaranteed by Proposition \ref{proponematrixisunique}.
In Proposition \ref{ProFullUniq1} and Corollaries \ref{Prop5.4}--\ref{corr:1.30}
overall CPD is derived from  uniqueness of two factor matrices, where the latter is guaranteed by Proposition \ref{proptwomatrixisunique}.
 We illustrate our results with some examples.

{\em Proof of Proposition \ref{ProFullUniq1onematrixWm}.}
 By \eqref{eq5.1}, $k_{\mathbf C}\geq 1$ and $\min(k_{\mathbf A},k_{\mathbf B})\geq m_{\mathbf C}-1$. Hence, by Proposition \ref{Prop:1.16}, $r_{\mathcal T}=R$ and the third factor matrix of $\mathcal T$ is unique. The result now follows from Proposition \ref{proponematrixisunique}.
\qquad\endproof

{\em Proof of Corollary \ref{ProFullUniq1onematrix}.}
From Proposition \ref{Prop:KCUW} (3) it follows that $\text{\textup{(W{\scriptsize m}}}_{\mathbf C}$$\text{\textup{)}}$ holds for $\mathbf A$, $\mathbf B$, and $\mathbf C$. Since \textup{(U{\scriptsize 1})} is equivalent to \textup{(C{\scriptsize 1})}, it follows from Proposition \ref{Prop:KCUW} (7) that $\mathbf A\odot\mathbf B$ has full column rank. The result now follows from Proposition \ref{ProFullUniq1onematrixWm}.
\qquad\endproof

{\em Proof of Corollaries \ref{ProFullUniq1onematrixHm} and \ref{ProFullUniq1onematrixcor1}.}
By Proposition \ref{Prop:KCUW} (2), both $\text{\textup{(H{\scriptsize m}}}_{\mathbf C}$$\text{\textup{)}}$ and $\text{\textup{(C{\scriptsize m}}}_{\mathbf C}$$\text{\textup{)}}$ imply $\text{\textup{(U{\scriptsize m}}}_{\mathbf C}$$\text{\textup{)}}$. The result now follows from Corollary \ref{ProFullUniq1onematrix}.\qquad\endproof

{\em Proof of Proposition \ref{ProFullUniq1}.}
Without loss of generality we assume that \textup{(i)} and \textup{(iii)} hold.
By  Proposition \ref{Prop:KCUW} (9),
\begin{equation}
\min(k_{\mathbf B},k_{\mathbf C})\geq m_{\mathbf A}\geq 2,\qquad \min(k_{\mathbf A},k_{\mathbf B})\geq m_{\mathbf C}\geq 2.
\label{minkakb}
\end{equation}
 It follows from   Proposition \ref{Prop:1.16} that $r_{\mathcal T}=R$ and that the  first and third factor matrices  of the tensor $\mathcal T$ are unique. One can easily check that \eqref{minkakb} implies \eqref{eq5.two}. Hence, by Proposition \ref{proptwomatrixisunique}, the CPD of $\mathcal T$ is unique.
\qquad\endproof

{\em Proof of Corollary \ref{Prop5.4HmHm}.}
Without loss of generality we assume that \textup{(ii)} and \textup{(iii)} hold. From
Proposition \ref{Prop:KCUW} (2) it follows that
\textup{(ii)} and \textup{(iii)} in Proposition \ref{ProFullUniq1} also hold.
Hence, by Proposition \ref{ProFullUniq1}, $r_{\mathcal T}=R$ and the CPD of $\mathcal T$ is unique.
\qquad\endproof

{\em Proof of Corollary \ref{Prop5.4}.}
By Proposition \ref{Prop:KCUW} (2), if two of the matrices in \eqref{cond13C} have full column rank, then at least two of conditions
\textup{(i)}--\textup{(iii)} in Proposition \ref{ProFullUniq1} hold. Hence, by Proposition \ref{ProFullUniq1}, $r_{\mathcal T}=R$ and the CPD of $\mathcal T$ is unique.
\qquad\endproof

{\em Proof of Corollary \ref{Proptwokkk}.}
Without loss of generality we assume that $(\mathbf X,\mathbf Y,\mathbf Z)=(\mathbf B,\mathbf C,\mathbf A)$. Then,
\begin{align*}
\begin{cases}
\left\{
\begin{array}{ll}
k_{\mathbf B}+r_{\mathbf A}+r_{\mathbf C}&\geq 2R+2,\\
k_{\mathbf A}+r_{\mathbf C}&\geq R+2,
\end{array}
\right.\\
\left\{
\begin{array}{ll}
k_{\mathbf B}+r_{\mathbf A}+r_{\mathbf C}&\geq 2R+2,\\
r_{\mathbf A}+k_{\mathbf C}&\geq R+2,
\end{array}
\right.\end{cases}\Rightarrow
\begin{cases}
\text{\textup{(K{\scriptsize m}}}_{\mathbf A} \text{\textup{)}}\   \text{ holds for }\ \mathbf B \ \text{ and }\ \mathbf C,\\
\text{\textup{(K{\scriptsize m}}}_{\mathbf C} \text{\textup{)}}\   \text{ holds for }\ \mathbf A \ \text{ and }\ \mathbf B,
\end{cases}
\end{align*}
where $m_{\mathbf A} = R - r_{\mathbf A} + 2$ and $m_{\mathbf C} = R - r_{\mathbf C} + 2$.
From Proposition \ref{Prop:KCUW} (1) it follows that the matrices
$\mathcal C_{m_{\mathbf A}}(\mathbf B)\odot \mathcal C_{m_{\mathbf A}}(\mathbf C)$ and
$\mathcal C_{m_{\mathbf C}}(\mathbf A)\odot \mathcal C_{m_{\mathbf C}}(\mathbf B)$
have full column rank. Hence, by  Corollary \ref{Prop5.4}, $r_{\mathcal T}=R$ and the CPD of $\mathcal T$ is unique.
\qquad\endproof

{\em Proof of Corollary \ref{corr:1.30}.}
It can be easily checked that all conditions of Corollary \ref{Proptwokkk} hold. Hence, $r_{\mathcal T}=R$ and the CPD of $\mathcal T$ is unique.
\qquad\endproof

\begin{example}\label{example5.6}
Consider a $5\times 5\times 5$ tensor given by the PD $\mathcal T=[\mathbf A, \mathbf B,  \mathbf C]_6$, where the matrices
$\mathbf A, \mathbf B, \mathbf C\in\mathbb C^{5\times 6}$  satisfy
$$
r_{\mathbf A}=r_{\mathbf B}=r_{\mathbf C}=5, \qquad k_{\mathbf A}=k_{\mathbf B}=k_{\mathbf C}=4.
$$
For instance, consider
$$
\mathbf A=
\left[
\begin{matrix}
1&0&0&0&0&* \\
0&1&0&0&0&* \\
0&0&1&0&0&* \\
0&0&0&1&0&* \\
0&0&0&0&1&0\\
\end{matrix}
\right], \
\mathbf B=
\left[
\begin{matrix}
1&0&0&0&0&* \\
0&1&0&0&0&* \\
0&0&1&0&0&* \\
0&0&0&1&0&0\\
0&0&0&0&1&* \\
\end{matrix}
\right]
, \
\mathbf C=
\left[
\begin{matrix}
1&0&0&0&0&* \\
0&1&0&0&0&* \\
0&0&1&0&0&0\\
0&0&0&1&0&* \\
0&0&0&0&1&* \\
\end{matrix}
\right],
$$
where $*$ denotes arbitrary nonzero entries.
Then Kruskal's condition \eqref{Kruskal} does not hold. On the other hand,  the conditions  of Corollary  \ref{Proptwokkk}
are satisfied. Hence, the PD of $\mathcal T$ is canonical and unique.
\end{example}
\begin{example}\label{example5.7}
Consider the $4\times 4\times 4$ tensor given by the PD $\mathcal T=[\mathbf A, \mathbf B,  \mathbf C]_5$,  where
$$
\mathbf A=
\left[
\begin{matrix}
1&0&0&0&1\\
0&1&0&0&1\\
0&0&1&0&1\\
0&0&0&1&0\\
\end{matrix}
\right], \quad
\mathbf B=
\left[
\begin{matrix}
1&0&0&0&1\\
0&1&0&0&1\\
0&0&1&0&0\\
0&0&0&1&1\\
\end{matrix}
\right], \quad
\mathbf C=
\left[
\begin{matrix}
1&0&0&0&1\\
0&1&0&0&0\\
0&0&1&0&1\\
0&0&0&1&1\\
\end{matrix}
\right].
$$
We have
$$
r_{\mathbf A}=r_{\mathbf B}=r_{\mathbf C}=4, \qquad k_{\mathbf A}=k_{\mathbf B}=k_{\mathbf C}=3,\qquad
m_{\mathbf A}=m_{\mathbf B}=m_{\mathbf C}=3.
$$
Hence, Kruskal's condition \eqref{Kruskal} does not hold.
Moreover, condition $\text{\textup{(K{\scriptsize 3})}}$ does not hold for  $(\mathbf A,\mathbf B)$,
$(\mathbf C,\mathbf A)$, nor $(\mathbf B,\mathbf C)$. Hence, the conditions of Corollary \ref{Proptwokkk}
 are not satisfied.
On the other hand, we have
\begin{align*}
\mathcal C_{3}(\mathbf A)\odot &\mathcal C_{3}(\mathbf B)\\
=&\left[
\begin{matrix}
\mathbf e^{16}_1&
\mathbf e^{16}_6&
\mathbf e^{16}_2&
\mathbf e^{16}_{11}&
\mathbf e^{16}_{1,-3}&
\mathbf e^{16}_{6,10}&
\mathbf e^{16}_{16}&
\mathbf e^{16}_{1,4}&
\mathbf e^{16}_{6,-14}&
\mathbf e^{16}_{11,12,15,16}
\end{matrix}
\right], \\
\mathcal C_{3}(\mathbf C)\odot &\mathcal C_{3}(\mathbf A)\\
=&\left[
\begin{matrix}
\mathbf e^{16}_1&
\mathbf e^{16}_6&
\mathbf e^{16}_{1,5}&
\mathbf e^{16}_{11}&
-\mathbf e^{16}_{9}&
\mathbf e^{16}_{10,11}&
\mathbf e^{16}_{16}&
\mathbf e^{16}_{1,13}&
\mathbf e^{16}_{6,16,-8,-14}&
\mathbf e^{16}_{11,12}
\end{matrix}
\right], \\
\mathcal C_{3}(\mathbf B)\odot &\mathcal C_{3}(\mathbf C)\\
=&\left[
\begin{matrix}
\mathbf e^{16}_1&
\mathbf e^{16}_6&
\mathbf e^{16}_{5,6}&
\mathbf e^{16}_{11}&
\mathbf e^{16}_{11,-3}&
\mathbf e^{16}_7&
\mathbf e^{16}_{16}&
\mathbf e^{16}_{1,4,13,16}&
\mathbf e^{16}_{6,-8}&
\mathbf e^{16}_{11,15}
\end{matrix}
\right],
\end{align*}
where
$$
\mathbf e^{16}_{i,\pm j}:=\mathbf e^{16}_i\pm\mathbf e^{16}_j,\quad
\mathbf e^{16}_{i,j,\pm k,\pm l}:=\mathbf e^{16}_i+\mathbf e^{16}_j\pm\mathbf e^{16}_k\pm\mathbf e^{16}_l,\quad
i,j,k,l\in\{1,\dots,16\}.
$$
It is easy to check that the matrices
$\mathcal C_{3}(\mathbf A)\odot \mathcal C_{3}(\mathbf B)$,  $\mathcal C_{3}(\mathbf C)\odot \mathcal C_{3}(\mathbf A)$ and
$\mathcal C_{3}(\mathbf B)\odot \mathcal C_{3}(\mathbf C)$ have full column rank.
Hence, by Corollary \ref{Prop5.4}, the PD is canonical and unique.

One can easily verify that $H_{\mathbf A\mathbf B}(\delta)=H_{\mathbf B\mathbf C}(\delta)=H_{\mathbf C\mathbf A}(\delta)=\min(\delta,3)$.
Hence the uniqueness of the CPD follows also from Corollary \ref{Prop5.4HmHm}.

Note that, since condition \eqref{eq5.1} does not hold, the result does not follow from Proposition
\ref{ProFullUniq1onematrixWm} and its Corollaries \ref{ProFullUniq1onematrix}--\ref{ProFullUniq1onematrixcor1}.
\end{example}

\begin{example}\label{OneHmbettertwoHm}
Consider the $5\times 5\times 8$ tensor given by the PD $\mathcal T=[\mathbf A, \mathbf B,  \mathbf C]_8$,  where
$$
\mathbf A=\left[\begin{matrix}\widehat{\mathbf A}\\ (\mathbf e_1^8)^T\end{matrix}\right]\in\mathbb F^{5\times 8},\qquad
\mathbf B=\left[\begin{matrix}\widehat{\mathbf B}\\ (\mathbf e_8^8)^T\end{matrix}\right]\in\mathbb F^{5\times 8},\qquad
 \mathbf C=\mathbf I_8
$$
and
$\widehat{\mathbf A}$ and $\widehat{\mathbf B}$ are $4\times 8$ matrices such that $k_{\widehat{\mathbf A}}=k_{\widehat{\mathbf B}}=4$.
We have $r_{\mathbf A}=r_{\mathbf B}=5$, $k_{\mathbf A}=k_{\mathbf B}=4$, and $r_{\mathbf C}=k_{\mathbf C}=8$.
One can easily check that
$$
H_{\mathbf A\mathbf B}(\delta)=
\begin{cases}
\delta,&1\leq\delta\leq 4,\\
3,&\delta=5,\\
2,&6\leq\delta\leq 8
\end{cases}\geq\min(\delta,8-8+2)
$$
and that condition \eqref{eq5.1} holds.
Hence, by Corollary \ref{ProFullUniq1onematrixHm}, the PD is canonical and unique. On the other hand,
$H_{\mathbf B\mathbf C}(\delta)=H_{\mathbf C\mathbf A}(\delta)=4<\min(\delta,8-5+2)$ for $\delta=5$.
Hence, the result does not follow from Corollary \ref{Prop5.4HmHm}.
\end{example}

\begin{example}\label{almostlast}
Let
$$
\mathbf A=\left[\begin{matrix}
1&0&0&1&1\\
0&1&0&1&2\\
0&0&1&1&3
\end{matrix}\right],\qquad
\mathbf B=\left[\begin{matrix}
1&0&0&1&1\\
0&1&0&1&3\\
0&0&1&1&5
\end{matrix}\right],\qquad \mathbf C=\mathbf I_5.
$$
It has already been shown in \cite{rankeqkrank2006} that the CPD of the tensor $\mathcal T=[\mathbf A,\mathbf B,\mathbf C]_5$ is unique. We give a shorter proof, based on Corollary \ref{ProFullUniq1onematrix}.
It is easy to verify that
\begin{align*}
\mathcal C_2(\mathbf A)\odot \mathcal C_2(\mathbf B)&=
\left[\begin{array}{rrrrrrrrrr}
1&	0&	1&	6 & 0&	 1&	  1&	0&	0&	2\\
0&	0&	1&	10& 0&	 0&	  0&	0&	0&	4\\
0&	0&	0&	0 & 0&	-1&	 -5&	0&	0&	2\\
0&	0&	1&	9 & 0&	 0&	  0&	0&	0&	4\\
0&	1&	1&	15& 0&	 0&	  0&	1&	1&	8\\
0&	0&	0&	0 & 0&	 0&	  0&	1&	3&	4\\
0&	0&	0&	0 & 0&	-1&	 -3&	0&	0&	2\\
0&	0&	0&	0 & 0&	 0&	  0&	1&	2&	4\\
0&	0&	0&	0 & 1&	 1&	  15&	1&	6&	2
\end{array}\right],\\
\textup{ker}(\mathcal C_2(\mathbf A)\odot \mathcal C_2(\mathbf B))&=\textup{span}\{\left[\begin{matrix}0& 0& -4& 0& 0& 2& 0& -4& 0& -1\end{matrix}\right]^T\}.
\end{align*}
If $\mathbf d\in \mathbb C^5$ is such that
$\textup{\text{diag}}(\mathcal C_2(\textup{\text{Diag}}(\mathbf d)))\in \textup{ker}(\mathcal C_2(\mathbf A)\odot \mathcal C_2(\mathbf B))$, we have
\begin{align*}
d_1d_2&=0,    &d_2d_3&= 0,  &d_3d_4&= -4c, &d_4d_5&=-c.\\
d_1d_3&=0,    &d_2d_4&= 2c, &d_3d_5&=  0,\\
d_1d_4&=-4c,  &d_2d_5&= 0,\\
d_1d_5&=0.
\end{align*}
One can check that this set of equations only has a solution if $c = 0$, in which case $\mathbf d= \mathbf 0$. Hence, by Corollary \ref{ProFullUniq1onematrix},
the PD is canonical and
unique.
Note that, since $m_{\mathbf A}=m_{\mathbf B}=5-3+2=  4$, the $m_{\mathbf A}$-th compound matrix of ${\mathbf A}$ and the $m_{\mathbf B}$-th compound matrix of ${\mathbf B}$ are not defined. Hence, the uniqueness of the matrices $\mathbf A$ and $\mathbf B$ does not follow from Proposition \ref{ProFullUniq1}.
\end{example}
\begin{example}\label{ex:77710}
Experiments indicate that for random $7\times 10$ matrices $\mathbf A$ and $\mathbf B$, the matrix
$\mathbf A\odot\mathbf B$ has full column rank and that  condition  $\text{\textup{(U{\scriptsize 5})}}$ does not hold.
Namely,  the kernel of the $441\times 252$ matrix
$\mathcal C_5(\mathbf A)\odot \mathcal C_5(\mathbf B)$  is spanned by  a  vector $\hatdSmR{5}{R}$ associated with some
$\mathbf d\in\mathbb F^{10}$. Let $\mathbf C$ be a  $7\times 10$ matrix such that $\mathbf d\not\in\textup{range}(\mathbf C^T)$.
Then $\text{\textup{(W{\scriptsize 5})}}$ holds for the triplet $(\mathbf A,\mathbf B,\mathbf C)$. If additionally
 $k_{\mathbf C}\geq 5$, then \eqref{eq5.1} holds. Hence, by Proposition \ref{ProFullUniq1onematrixWm},
  $r_{\mathcal T}=10$ and the CPD of $\mathcal T=[\mathbf A,\mathbf B,\mathbf C]_{10}$ is unique.

The same situation occurs for tensors with other dimensions (see Table \ref{tab:condUmWm}).
\begin{table}[htbp]
\caption{
Some cases where the rank and  the uniqueness of the CPD of $\mathcal T=[\mathbf A,\mathbf B,\mathbf C]_R$
may be easily obtained from  Proposition \ref{ProFullUniq1onematrixWm} or its Corollary \ref{ProFullUniq1onematrix} (see Example \ref{ex:77710}).
Matrices $\mathbf A$, $\mathbf B$ and  $\mathbf C$ are generated randomly. Simulations indicate
that the dimensions of  $\mathbf A$ and $\mathbf B$ cause the dimension
of $\textup{ker}(\mathcal C_m(\mathbf A)\odot\mathcal C_m(\mathbf B))$ to be equal to $1$. Thus,
\cond{U}{m} and \cond{W}{m} may be easily checked.
}
\begin{center}\footnotesize
\begin{tabular}{|c|c|c|c|l|l|}
  \hline
   dimensions of  $\mathcal T$, &$r_{\mathcal T}=R$     & m=R-K+2 & dimensions of                                         &\cond{U}{m}  &\cond{W}{m}  \\
   $I\times J\times K$         &                       &         &$\mathcal C_m(\mathbf A)\odot\mathcal C_m(\mathbf B)$&             &\\
\hline
 $4\times 5\times 6$          & $7$                   & $3$     &$40\times 35$                                        &does not hold &holds \\
 $4\times 6\times 14$         & $14$                  & $2$     &$90\times 91$                                         &holds &holds  \\
 $5\times 7\times 7$          & $9$                   & $4$     &$175\times 216$                                         &does not hold &holds \\
 $6\times 9\times 8$          & $11$                  & $5$     &$756\times 462$                                         &does not hold &holds \\
 $7\times 7\times 7$          & $10$                  & $5$     &$441\times 252$                                         &does not hold &holds \\
     \hline
\end{tabular}
\end{center}
\label{tab:condUmWm}
\end{table}
\end{example}

\section{Application to tensors with symmetric frontal slices and  Indscal} \label{sect:symm}

In this section we consider  tensors with  symmetric frontal slices (SFS), which we will briefly call SFS-tensors. We are interested  in
PDs of which the rank-$1$ terms have the same symmetry. Such decompositions correspond to the INDSCAL model, as introduced by Carroll and Chang \cite{1970_Carroll_Chang}. A similar approach may be followed in the case of full symmetry.

We start with definitions of SFS-rank, SFS-PD, and SFS-CPD.

\begin{definition}\label{Def: outer product:indscal}
A third-order SFS-tensor $\mathcal T\in\mathbb F^{I\times I\times K}$ is {\em SFS-rank-$1$}
if it equals the outer product of three nonzero vectors $\mathbf a\in\mathbb F^I$,
$\mathbf a\in\mathbb F^I$ and $\mathbf c\in\mathbb F^K$.
\end{definition}

\begin{definition}
A {\em SFS-PD} of a third-order SFS-tensor $\mathcal T\in\mathbb F^{I\times I\times K}$ expresses $\mathcal T$ as a sum of  SFS-rank-$1$ terms:
\begin{equation}
\mathcal T=\sum\limits_{r=1}^R\mathbf a_r\circ \mathbf a_r\circ \mathbf c_r,
 \label{eqintro2:indscal}
\end{equation}
where $\mathbf a_r \in \mathbb F^{I}$,  $\mathbf c_r \in \mathbb F^{K}$, $1 \leq r \leq R$.
\end{definition}

\begin{definition}
The {\em SFS-rank} of a SFS-tensor $\mathcal T \in\mathbb F^{I\times I\times K}$ is defined
as the minimum number of SFS-rank-$1$ tensors in a PD  of  $\mathcal T$ and is denoted by $r_{SFS,\mathcal T}$.
\end{definition}

\begin{definition}
A {\em SFS-CPD}  of a third-order SFS-tensor $\mathcal T$ expresses $\mathcal T$ as a minimal sum of SFS-rank-$1$ terms.
\end{definition}

Note that $\mathcal T=[\mathbf A,\mathbf B,\mathbf C]_R$ is a SFS-CPD of $\mathcal T$ if and only if $\mathcal T$ is an SFS-tensor, $\mathbf A=\mathbf B$, and $R=r_{SFS,\mathcal T}$.

Now we can define  uniqueness of the SFS-CPD.

\begin{definition}\label{Def:1.5:indscal}
Let $\mathcal T$ be a SFS-tensor of SFS-rank $R$. The SFS-CPD  of  $\mathcal T$ is {\em unique} if
$\mathcal T=[\mathbf A,\mathbf A,\mathbf C]_R=[\bar{\mathbf A},\bar{\mathbf A},\bar{\mathbf C}]_R$ implies that there exist an $R\times R$ permutation matrix $\mathbf{\Pi}$ and $R\times R$ nonsingular diagonal matrices ${\mathbf \Lambda}_{\mathbf A}$ and ${\mathbf \Lambda}_{\mathbf C}$ such that
$$
\bar{\mathbf A}=\mathbf A\mathbf{\Pi}{\mathbf \Lambda}_{\mathbf A},\quad
\bar{\mathbf C}=\mathbf C\mathbf{\Pi}{\mathbf \Lambda}_{\mathbf C},\quad
{\mathbf \Lambda}_{\mathbf A}^2{\mathbf \Lambda}_{\mathbf C}=\mathbf I_R.
$$
\end{definition}
\begin{example}\label{ex:indscal}
Some SFS-tensors admit both SFS-CPDs and CPDs of which the terms are not partially symmetric. For instance,
consider the SFS-tensor $\mathcal T \in\mathbb R^{I\times I\times K}$ in which $\mathbf I_I$ is stacked $K$ times.
Let $\mathbf E$ denote the $K\times I$ matrix of which all entries are equal to one. Then
 $
 \mathcal T=[\mathbf X,(\mathbf X^{-1})^T,\mathbf E]_I,
 $
 is a CPD of $ \mathcal T$ for any nonsingular $I\times I$ matrix $\mathbf X$. On the other hand,
 $
 \mathcal T=[\mathbf A,\mathbf A,\mathbf E]_I,
 $
 is a SFS-CPD of $\mathcal T$ for any orthogonal $I\times I$ matrix $\mathbf A$.
\end{example}

The following result was obtained in \cite{tenBergSidRoccie2004}.
We present the proof for completeness.
\begin{lemma}\label{Lemma:5.13}
Let $\mathcal T$ be a SFS-tensor of rank $R$ and let the CPD of $\mathcal T$ be unique.
Then $r_{SFS,\mathcal T}=r_{\mathcal T}$, and the SFS-CPD  of  $\mathcal T$ is also unique.
\end{lemma}
\begin{proof}
Let $[\mathbf A,\mathbf B,\mathbf C]_R$ be a CPD of the SFS-tensor $\mathcal T$. Because of the symmetry we also have $\mathcal T=[\mathbf B,\mathbf A,\mathbf C]_R$. Since the CPD of
$\mathcal T$ is unique,
there exist an $R\times R$ permutation matrix $\mathbf{\Pi}$ and $R\times R$ nonsingular diagonal matrices
${\mathbf \Lambda}_{\mathbf A}$, ${\mathbf \Lambda}_{\mathbf B}$, and ${\mathbf \Lambda}_{\mathbf C}$ such that
$\mathbf B=\mathbf A\mathbf{\Pi}{\mathbf \Lambda}_{\mathbf A}$, $\mathbf A=\mathbf B\mathbf{\Pi}{\mathbf \Lambda}_{\mathbf B}$,
$\mathbf C=\mathbf C\mathbf{\Pi}{\mathbf \Lambda}_{\mathbf C}$ and $\Lambda_{\mathbf A}\Lambda_{\mathbf B}{\mathbf \Lambda}_{\mathbf C}=\mathbf I_R$.
Since the CPD is unique, by \eqref{nec1}, we have $k_{\mathbf C}\geq 2$. Hence, $\mathbf{\Pi}={\mathbf \Lambda}_{\mathbf C}=\mathbf I_R$ and
$\mathbf B=\mathbf A{\mathbf \Lambda}_{\mathbf A}$.
Thus, any CPD of $\mathcal T$ is in fact a SFS-CPD.
Hence, $r_{SFS,\mathcal T}=r_{\mathcal T}$, and the SFS-CPD  of  $\mathcal T$ is unique.
\end{proof}
\begin{remark}
To the authors' knowledge, it is still an open question whether there exist SFS-tensors with unique SFS-CPD but non-unique CPD.
\end{remark}

Lemma \ref{Lemma:5.13} implies that conditions guaranteeing uniqueness of SFS-CPD may be obtained from conditions guaranteeing uniqueness of CPD
by just ignoring the SFS-structure.  To illustrate this,
we present SFS-variants of Corollaries \ref{ProFullUniq1onematrixcor1} and \ref{Prop5.4}.

\begin{proposition}\label{ProFullUniq1onematrixcor1:indscal12}
Let $\mathcal T=[\mathbf A, \mathbf A, \mathbf C]_R$ and
$m_{\mathbf C}:=R-r_{\mathbf C}+2$.
Assume that
\begin{romannum}
\item
$k_{\mathbf A}+k_{\mathbf C}\geq R+2$;
\item
$\mathcal C_{m_{\mathbf C}}(\mathbf A)\odot \mathcal C_{m_{\mathbf C}}(\mathbf A)$ has full column rank.
\end{romannum}
Then $r_{SFS,\mathcal T}=R$ and the SFS-CPD of  tensor $\mathcal T$  is unique.
\end{proposition}
\begin{proof}
From Corollary  \ref{ProFullUniq1onematrixcor1} it follows that
$r_{\mathcal T}=R$ and that the CPD of  tensor $\mathcal T$  is unique.
The proof now follows from Lemma \ref{Lemma:5.13}.
\end{proof}
\begin{remark}
Under the additional assumption $r_{\mathbf C}=R$, Proposition \ref{ProFullUniq1onematrixcor1:indscal12}
was proved in \cite{Stegeman2009}.
\end{remark}
\begin{proposition}\label{ProFullUniq1onematrixcor1:indscal3}
Let $\mathcal T=[\mathbf A, \mathbf A, \mathbf C]_R$ and
$m_{\mathbf A}:=R-r_{\mathbf A}+2$.
Assume that
\begin{romannum}
\item
$k_{\mathbf A}+\max(\min(k_{\mathbf C}-1,k_{\mathbf A}),\min(k_{\mathbf C},k_{\mathbf A}-1))\geq R+1$;
\item
$\mathcal C_{m_{\mathbf A}}(\mathbf A)\odot \mathcal C_{m_{\mathbf A}}(\mathbf C)$ has full column rank.
\end{romannum}
Then $r_{SFS,\mathcal T}=R$ and the SFS-CPD of  tensor $\mathcal T$  is unique.
\end{proposition}
\begin{proof}
By Lemma \ref{Lemma:5.13} it is sufficient to show that
$r_{\mathcal T}=R$ and that the CPD of  tensor $\mathcal T$  is unique. Both these statements follow from
Corollary  \ref{ProFullUniq1onematrixcor1} applied to the tensor $[\mathbf A,\mathbf C,\mathbf A]_R$.
\end{proof}
\begin{proposition}\label{Prop5.4:indscal}
Let $\mathcal T=[\mathbf A, \mathbf A, \mathbf C]_R$, $m_{\mathbf A}=R-r_{\mathbf A}+2$ and   $m_{\mathbf C}=R-r_{\mathbf C}+2$.
Assume that the  matrices
\begin{gather}
\mathcal C_{m_{\mathbf A}}(\mathbf A)\odot \mathcal C_{m_{\mathbf A}}(\mathbf C), \label{cond1C:indscal}\\
\mathcal C_{m_{\mathbf C}}(\mathbf A)\odot \mathcal C_{m_{\mathbf C}}(\mathbf A)\label{cond3C:indscal}
\end{gather}
have full column rank.
Then $r_{SFS,\mathcal T}=R$ and the SFS-CPD of  tensor $\mathcal T$  is unique.
\end{proposition}
\begin{proof}
From Corollary  \ref{Prop5.4} it follows that
$r_{\mathcal T}=R$ and that the CPD of  tensor $\mathcal T$  is unique.
The proof now follows from Lemma \ref{Lemma:5.13}.
\end{proof}
\section{Uniqueness beyond $\text{\textup{(W{\scriptsize m})}}$} \label{Appendix Extra}

In this section we discuss an example in which even condition $\text{\textup{(W{\scriptsize m})}}$ is not satisfied. Hence, CPD uniqueness does not follow from Proposition \ref{prmostgeneraldis} or Proposition \ref{Prop:1.16}. A fortiori, it does not follow from Proposition   \ref{ProFullUniq1onematrixWm}, Corollaries \ref{ProFullUniq1onematrix}--\ref{ProFullUniq1onematrixcor1},  Proposition \ref{ProFullUniq1}, and  Corollaries \ref{Prop5.4}--\ref{corr:1.30}. We show that uniqueness of the CPD can nevertheless be demonstrated by combining subresults.

In this section
we will denote by $\omega(\mathbf d)$ the number of nonzero components of $\mathbf d$ and we will
 write $\mathbf a\parallel\mathbf b$ if the vectors $\mathbf a$ and $\mathbf b$ are collinear, that is there exists a nonzero number $c\in\mathbb F$ such that $\mathbf a=c\mathbf b$.

For easy reference we include the following lemma concerning second compound matrices.
\begin{lemma}\label{BCformula_and_diag} \cite[Lemma 2.4 (1) and Lemma 2.5]{Part I}
\begin{itemize}
\item[\textup{(1)}]
Let the product $\mathbf X\mathbf Y\mathbf Z$ be defined. Then the product $\mathcal C_2(\mathbf X)\mathcal C_2(\mathbf Y)\mathcal C_2(\mathbf Z)$
is also defined and
$$
\mathcal C_2(\mathbf X\mathbf Y\mathbf Z)=\mathcal C_2(\mathbf X)\mathcal C_2(\mathbf Y)\mathcal C_2(\mathbf Z).
$$
\item[\textup{(2)}]
Let $\mathbf d=\left[\begin{matrix}d_1&d_2&\dots&d_R\end{matrix}\right]\in\mathbb F^R$.
Then $\mathcal C_2(\textup{\text{Diag}}(\mathbf d))=\textup{\text{Diag}}(\hatdSmR{2}{R})$.

In particular, $\omega(\mathbf d)\leq 1$ if and only if $\hatdSmR{2}{R}=\vzero$ if and only if  $\mathcal C_2(\textup{\text{Diag}}(\mathbf d))=\vzero$.
\end{itemize}
\end{lemma}

\begin{example}
Let $\mathcal T_\alpha=[\mathbf A,\mathbf B,\mathbf C]_5$, where
$$
\mathbf A=
\left[
\begin{matrix}
0&	\alpha&	0&	0&	0\\
1&	0&	1&	0&	0\\
1&	0&	0&	1&	0\\
0&	0&	0&	0&	1
\end{matrix}
\right],\
\mathbf B=
\left[
\begin{matrix}
0&	1&	0&	0&	0\\
1&	0&	1&	0&	0\\
0&	0&	0&	1&	0\\
1&	0&	0&	0&	1
\end{matrix}
\right],\
\mathbf C=
\left[
\begin{matrix}
1&	1&	0&	0&	0\\
0&	0&	1&	0&	0\\
0&	0&	0&	1&	0\\
1&	0&	0&	0&	1
\end{matrix}
\right],\ \alpha\ne 0.
$$
Then
$r_{\mathbf A}=r_{\mathbf B}=r_{\mathbf C}=4$, $k_{\mathbf A}=k_{\mathbf B}=k_{\mathbf C}=2$, and $m:=m_{\mathbf A}=m_{\mathbf B}=m_{\mathbf C}=5-4+2=3$.
One can  check that none of the triplets $(\mathbf A,\mathbf B,\mathbf C)$, $(\mathbf B,\mathbf C,\mathbf A)$, $(\mathbf C,\mathbf A,\mathbf B)$
satisfies condition $\text{\textup{(W{\scriptsize m})}}$. Hence, the rank and the uniqueness of the factor matrices of $\mathcal T_\alpha$  do not follow from Proposition \ref{prmostgeneraldis} or Proposition \ref{Prop:1.16}.
We prove that $r_{\mathcal T_{\alpha}}=5$ and that the CPD  $\mathcal T_\alpha=[\mathbf A,\mathbf B,\mathbf C]_5$   is  unique.
\begin{romannum}
\item A trivial verification shows that
\begin{align}
&\mathbf A\odot \mathbf B,\quad
\mathbf B\odot \mathbf C,\quad
\mathbf C\odot \mathbf A, &\text{have full column rank,}\\
&\mathcal C_2(\mathbf A)\odot \mathcal C_2(\mathbf B),\quad
\mathcal C_2(\mathbf B)\odot \mathcal C_2(\mathbf C),\quad
\mathcal C_2(\mathbf C)\odot \mathcal C_2(\mathbf A)  &\text{have full column rank.}\label{eq7.12}
\end{align}
Elementary algebra yields
\begin{alignat}{4}
\omega(\mathbf A^T\mathbf x)=1  &\Leftrightarrow \ \mathbf x&\parallel\mathbf e_1^4\ &\text{ or }&\ \mathbf x&\parallel\mathbf e_4^4,\label{eq7.4a}\\
\omega(\mathbf B^T\mathbf y)=1  &\Leftrightarrow \ \mathbf y&\parallel\mathbf e_1^4\ &\text{ or }&\ \mathbf y&\parallel\mathbf e_3^4,\label{eq7.4b}\\
\omega(\mathbf C^T\mathbf z)=1  &\Leftrightarrow \ \mathbf z&\parallel\mathbf e_2^4\ &\text{ or }&\ \mathbf z&\parallel\mathbf e_3^4.\label{eq7.4c}
\end{alignat}
\item
Consider a CPD $\mathcal T_{\alpha}=[\bar{\mathbf A},\bar{\mathbf B},\bar{\mathbf C}]_{\bar R}$, i.e. ${\bar R}=r_{\mathcal T_{\alpha}}$ is minimal. We have  ${\bar R}\leq 5$.
For later use we show that any three solutions of the equation  $\omega(\bar{\mathbf C}^T\mathbf z)=1$ are linearly dependent. Indeed, assume that there exist three vectors $\mathbf z_1,\ \mathbf z_2,\ \mathbf z_3\in\mathbb F^4$ such that
$\omega(\bar{\mathbf C}^T\mathbf z_1)=\omega(\bar{\mathbf C}^T\mathbf z_2)=\omega(\bar{\mathbf C}^T\mathbf z_3)=1$.
 By \eqref{eqT_V}--\eqref{eqT_V2},
\begin{align}
\mathbf T^{(1)}&=(\bar{\mathbf A}\odot\bar{\mathbf B})\bar{\mathbf C}^T=(\mathbf A\odot\mathbf B)\mathbf C^T,\label{matrreprdisc1}\\
\mathbf T^{(2)}&=(\bar{\mathbf B}\odot\bar{\mathbf C})\bar{\mathbf A}^T=(\mathbf B\odot\mathbf C)\mathbf A^T,\label{matrreprdisc12}\\
\mathbf T^{(3)}&=(\bar{\mathbf C}\odot\bar{\mathbf A})\bar{\mathbf B}^T=(\mathbf C\odot\mathbf A)\mathbf B^T. \label{matrreprdisc13}
\end{align}
From \eqref{matrreprdisc1}  it follows that
$\mathbf A\textup{Diag}(\mathbf C^T\mathbf z_i)\mathbf B^T=\bar{\mathbf A}\textup{Diag}(\bar{\mathbf C}^T\mathbf z_i)\bar{\mathbf B}^T$, and hence, by
Lemma \ref{BCformula_and_diag} (1),
$$
\mathcal C_2(\mathbf A)\mathcal C_2(\textup{Diag}(\mathbf C^T\mathbf z_i))\mathcal C_2(\mathbf B^T)=
\mathcal C_2(\bar{\mathbf A})\mathcal C_2(\textup{Diag}(\bar{\mathbf C}^T\mathbf z_i))\mathcal C_2(\bar{\mathbf B}^T)=\mzero,\qquad i\in\{1,2,3\},
$$
which can also be expressed as
$$
\left[\mathcal C_2(\mathbf A)\odot\mathcal C_2(\mathbf B)\right]\hatdSmR{2}{R}_i=\vzero,\qquad
\mathbf d_i:=\mathbf C^T\mathbf z_i,\qquad i\in\{1,2,3\}.
$$
By \eqref{eq7.12},
$\hatdSmR{2}{R}_i=\vzero$ for $i\in\{1,2,3\}$.
Since $\mathbf C^T$ has full column rank,  Lemma  \ref{BCformula_and_diag} (2) implies that
$\omega(\mathbf C^T\mathbf z_1)=\omega(\mathbf C^T\mathbf z_2)=\omega(\mathbf C^T\mathbf z_3)=1$. From \eqref{eq7.4c} it follows that
at least two of the vectors $\mathbf z_1$, $\mathbf z_2$, and $\mathbf z_3$ are collinear. Hence, the vectors
$\mathbf z_1$, $\mathbf z_2$, and $\mathbf z_3$ are linearly dependent.
\item
Since $\mathbf A\odot \mathbf B$ and $\mathbf C^T$ have full column rank, from \eqref{matrreprdisc1} and Sylvester's rank inequality it follows that
$$
r_{\bar{\mathbf C}^T}\geq r_{(\bar{\mathbf A}\odot\bar{\mathbf B})\bar{\mathbf C}^T}=
r_{({\mathbf A}\odot{\mathbf B}){\mathbf C}^T}\geq r_{{\mathbf A}\odot{\mathbf B}}+r_{{\mathbf C}^T}-5=5+4-5=4.
$$
In a similar fashion, from \eqref{matrreprdisc12} and \eqref{matrreprdisc13} we obtain
$r_{\bar{\mathbf A}^T}\geq 4$ and $r_{\bar{\mathbf B}^T}\geq 4$, respectively. We conclude that ${\bar R}\geq r_{\bar{\mathbf A}}=r_{\bar{\mathbf B}}=r_{\bar{\mathbf C}}=4$.

Let us show that ${\bar R}=5$. To obtain a contradiction, assume that ${\bar R}=4$. In this case, since $r_{\bar{\mathbf C}^T}=4$, $\bar{\mathbf C}^T$ is a nonsingular square matrix.
Then the columns of $\mathbf Z:=(\bar{\mathbf C}^T)^{-1}$ are linearly independent solutions of the equation
$\omega(\bar{\mathbf C}^T\mathbf z)=1$, which is a contradiction with $\textup{(ii)}$. Hence, ${\bar R}=5$.
\item
Let us show that $k_{\bar{\mathbf C}}\geq 2$. Conversely, assume that  $k_{\bar{\mathbf C}}=1$.
Since $r_{\bar{\mathbf C}}=4$, it follows that there exists exactly one pair of proportional columns of $\bar{\mathbf C}$. Without loss of generality we will assume that
$\bar{\mathbf C}_4\parallel\bar{\mathbf C}_5$. Hence,
$\mathbb F^4=\textup{range}(\bar{\mathbf C})=\textup{span}\{\bar{\mathbf C}_1,\bar{\mathbf C}_2,\bar{\mathbf C}_3,\bar{\mathbf C}_4\}$.
Let $\left[\begin{matrix}\mathbf z_1&\mathbf z_2&\mathbf z_3&\mathbf z_4\end{matrix}\right]:=(\left[\begin{matrix}\bar{\mathbf C}_1&\bar{\mathbf C}_2&\bar{\mathbf C}_3&\bar{\mathbf C}_4\end{matrix}\right]^T)^{-1}$.
Then
$\omega(\bar{\mathbf C}^T\mathbf z_1)=\omega(\bar{\mathbf C}^T\mathbf z_2)=\omega(\bar{\mathbf C}^T\mathbf z_3)=1$, which is a contradiction with $\textup{(ii)}$.

In a similar fashion we can prove that $k_{\bar{\mathbf A}}\geq 2$ and $k_{\bar{\mathbf B}}\geq 2$. Thus, $\min (k_{\bar{\mathbf A}}, k_{\bar{\mathbf B}},k_{\bar{\mathbf C}})\geq 2$.
\item Assume that
 there exist indices $i,j,k,l$ and nonzero values $t_1,t_2,t_3,t_4$ such that
\begin{equation}\label{matrreprdisc2}
(\bar{\mathbf A}^T)_1=t_1\mathbf e_i^5,\quad
(\bar{\mathbf A}^T)_4=t_2\mathbf e_j^5,\quad
(\bar{\mathbf B}^T)_1=t_3\mathbf e_k^5,\quad
(\bar{\mathbf B}^T)_3=t_4\mathbf e_l^5.
\end{equation}
Here we show that \eqref{matrreprdisc2} implies the
uniqueness of the CPD of $\mathcal T_{\alpha}$ and a fortiori the uniqueness of the third factor matrix. The latter implication will as such be instrumental in the proof of \textup{(vi)}. That assumption  \eqref{matrreprdisc2} really holds, and thus implies CPD uniqueness, will be demonstrated in \textup{(vii)}.

Combination of \eqref{matrreprdisc12}, \eqref{matrreprdisc13}, and \eqref{matrreprdisc2} yields
\begin{align*}
&\alpha\mathbf b_2\otimes\mathbf c_2 =t_1\bar{\mathbf b}_i\otimes\bar{\mathbf c}_i,&
&\mathbf c_2\otimes\mathbf a_2 =t_3\bar{\mathbf c}_k\otimes\bar{\mathbf a}_k,&\\
&\ \ \mathbf b_5\otimes\mathbf c_5 =t_2\bar{\mathbf b}_j\otimes\bar{\mathbf c}_j,&
&\mathbf c_4\otimes\mathbf a_4 =t_4\bar{\mathbf c}_l\otimes\bar{\mathbf a}_l.&
\end{align*}
We see that $\mathbf b_2\parallel\bar{\mathbf b}_i$. Also, $\mathbf c_2\parallel\bar{\mathbf c}_i$ and  $\mathbf c_2\parallel\bar{\mathbf c}_k$.
Since $k_{\bar{\mathbf C}}\geq 2$, it follows that $i=k$. Therefore, also
$\mathbf a_2\parallel\bar{\mathbf a}_i$.
It is now clear that  $[\mathbf a_2,\mathbf b_2,\mathbf c_2]_1-[\bar{\mathbf a}_i,\bar{\mathbf b}_i,\bar{\mathbf c}_i]_1=\beta[\mathbf e_1^4,\mathbf e_1^4,\mathbf e_1^4]$ for some $\beta\in\mathbb F$.
Let
$$
\mathcal T_\beta:=
\mathcal T_\alpha-[\mathbf a_2,\mathbf b_2,\mathbf c_2]_1+\beta[\mathbf e_1^4,\mathbf e_1^4,\mathbf e_1^4]=
[\bar{\mathbf A},\bar{\mathbf B},\bar{\mathbf C}]_5-[\bar{\mathbf a}_i,\bar{\mathbf b}_i,\bar{\mathbf c}_i]_1.
$$
Obviously, $\mathcal T_\beta$ is rank-$4$.
We claim that $\beta=0$. Indeed, if $\beta\ne 0$, then repeating steps \textup{(i)}--\textup{(iii)} for $\mathcal T_\alpha$ replaced by $\mathcal T_\beta$ we obtain that
$\mathcal T_\beta$ is rank-$5$, which is a contradiction.
Hence,
$[\mathbf a_2,\mathbf b_2,\mathbf c_2]_1=[\bar{\mathbf a}_i,\bar{\mathbf b}_i,\bar{\mathbf c}_i]_1$.

What is left to show, is that the CPD of the rank-$4$ tensor $\mathcal T_\alpha-[\mathbf a_2,\mathbf b_2,\mathbf c_2]_1$ is unique.
Note that the matrix
$\left[\begin{matrix}\mathbf c_1&\mathbf c_3&\mathbf c_4&\mathbf c_5\end{matrix}\right]$ has full column rank.
From \eqref{eq7.12} it follows that
$\mathcal C_2(\left[\begin{matrix}\mathbf a_1&\mathbf a_3&\mathbf a_4&\mathbf a_5\end{matrix}\right])
\odot \mathcal C_2( \left[\begin{matrix}\mathbf b_1&\mathbf b_3&\mathbf b_4&\mathbf b_5\end{matrix}\right])$ also has full column rank. Hence, by Proposition \ref{K2C2U2unique}, the CPD of\ $\mathcal T_\alpha-[\mathbf a_2,\mathbf b_2,\mathbf c_2]_1$ is unique.

\item  Let us show that $k_{\bar{\mathbf A}}=k_{\bar{\mathbf B}}=k_{\bar{\mathbf C}}= 2$. Conversely, assume that $k_{\bar{\mathbf C}}\geq 3$.
Then
$r_{\bar{\mathbf B}}+k_{\bar{\mathbf C}}\geq 4+3\geq R+2$.
Recall from (iv) that $k_{\bar{\mathbf B}}\geq 2$.
Hence, condition (\text{K{\scriptsize 2}}) holds for
$\bar{\mathbf B},\bar{\mathbf C}$.
By Proposition \ref{Prop:KCUW} (1),
\begin{equation} \label{B.6}
\mathcal C_2(\bar{\mathbf B})\odot\mathcal C_2(\bar{\mathbf C}) \text{ has full column rank.}
\end{equation}
Let $\mathbf x\in\mathbb F^4$. From \eqref{matrreprdisc12} it follows that
$$
\mathbf B\textup{Diag}(\mathbf A^T\mathbf x)\mathbf C^T=\bar{\mathbf B}\textup{Diag}(\bar{\mathbf A}^T\mathbf x)\bar{\mathbf C}^T,
$$
Hence, by
 Lemma \ref{BCformula_and_diag} (1),
\begin{equation*}
\mathcal C_2(\mathbf B)\mathcal C_2(\textup{Diag}(\mathbf A^T\mathbf x))\mathcal C_2(\mathbf C^T)=
\mathcal C_2(\bar{\mathbf B})\mathcal C_2(\textup{Diag}(\bar{\mathbf A}^T\mathbf x))\mathcal C_2(\bar{\mathbf C}^T),
\end{equation*}
which can also be expressed as
\begin{equation}
\left[\mathcal C_2(\mathbf B)\odot\mathcal C_2(\mathbf C)\right]\widehat{\mathbf d}^2_{\mathbf A}=
\left[\mathcal C_2(\bar{\mathbf B})\odot\mathcal C_2(\bar{\mathbf C})\right]\widehat{\mathbf d}^2_{\bar{\mathbf A}},\label{eqappb1}
\end{equation}
where
$\mathbf d_{\mathbf A}=\mathbf A^T\mathbf x$ and
$\mathbf d_{\bar{\mathbf A}}=\bar{\mathbf A}^T\mathbf x$.
From \eqref{eq7.12}, \eqref{B.6}, and  Lemma  \ref{BCformula_and_diag} (2) it follows that
\begin{equation}
\begin{split}
\omega(\mathbf A^T\mathbf x)=1&\xLeftrightarrow{\text{Lemma  \ref{BCformula_and_diag} (2)}}
 \widehat{\mathbf d}^2_{\mathbf A}=\vzero\xLeftrightarrow{\eqref{eq7.12}, \eqref{B.6},\eqref{eqappb1}}
  \widehat{\mathbf d}^2_{\bar{\mathbf A}}=\vzero\\
  &\xLeftrightarrow{\text{Lemma  \ref{BCformula_and_diag} (2)}}
  \omega(\bar{\mathbf A}^T\mathbf x)=1.\label{eqappb5}
  \end{split}
\end{equation}
In a similar fashion we can prove that for $\mathbf y\in\mathbb F^4$,
\begin{equation}
\omega(\mathbf B^T\mathbf y)=1\Leftrightarrow
  \omega(\bar{\mathbf B}^T\mathbf y)=1.\label{eqappb6}
\end{equation}
Therefore, by \textup{(i)}, there exist indices $i,j,k,l$ and nonzero values $t_1,t_2,t_3,t_4$ such that
\eqref{matrreprdisc2} holds. It follows from step $\textup{(v)}$ that the matrices $\mathbf C$ and $\bar{\mathbf C}$ are the same up
to permutation and column scaling. Hence, $k_{\bar{\mathbf C}}=k_{\mathbf C}=2$, which is a contradiction with $k_{\bar{\mathbf C}}\geq 3$. We conclude that $k_{\bar{\mathbf C}} < 3$. On the other hand, we have from (iv) that $k_{\bar{\mathbf C}} \geq 2$. Hence, $k_{\bar{\mathbf C}} = 2$.

In a similar fashion we can prove that $k_{\bar{\mathbf A}}=k_{\bar{\mathbf B}}=2$.

\item Since $k_{\bar{\mathbf A}}=k_{\bar{\mathbf B}}=2$, both $\bar{\mathbf A}$ and ${\bar{\mathbf B}}$ have a rank-deficient $4 \times 3$ submatrix.
 Since $r_{\bar{\mathbf A}}=r_{\bar{\mathbf B}}=4$, it follows that there exist vectors $\mathbf x_1$, $\mathbf x_2$, $\mathbf y_1$, $\mathbf y_2$ such that
    $$
    \omega(\bar{\mathbf A}^T\mathbf x_1)=\omega(\bar{\mathbf A}^T\mathbf x_2)=
    \omega(\bar{\mathbf B}^T\mathbf y_1)=\omega(\bar{\mathbf B}^T\mathbf y_2)=1,\qquad
    \mathbf x_1\not\parallel\mathbf x_2,\qquad
    \mathbf y_1\not\parallel\mathbf y_2.
    $$
From \eqref{eqappb1}--\eqref{eqappb6} it follows  that $\omega(\mathbf A^T\mathbf x_1)=\omega(\mathbf A^T\mathbf x_2)=
    \omega(\mathbf B^T\mathbf y_1)=\omega(\mathbf B^T\mathbf y_2)=1$.  By \eqref{eq7.4a}--\eqref{eq7.4b} there exist
    indices $i,j,k,l$ and nonzero values $t_1,t_2,t_3,t_4$ such that
\eqref{matrreprdisc2} holds. Hence, by \textup{(v)}, the CPD of $\mathcal T_\alpha$ is unique.
\end{romannum}
\end{example}

\section{Generic uniqueness}\label{Section6}
\subsection{Generic uniqueness of unconstrained CPD}\label{sec:last1}

 It was explained in \cite{DeLathauwer2006, Psycho2006} that the conditions
$r_{\mathbf C}=R$ and (\textup{C{\scriptsize 2}}) in Proposition \ref{K2C2U2unique}
hold generically when they hold for one particular choice of $\mathbf A$, $\mathbf B$ and $\mathbf C$. It was indicated that this implies that the CPD of an $I\times J\times K$ tensor $\mathcal T=[\mathbf A, \mathbf B, \mathbf C]_R$
is generically unique whenever $K\geq R$ and $C^2_IC^2_J\geq C^2_R$.
These conditions guarantee that the matrix $\mathbf C$ generically has full column rank and that the number of columns of
the $C^2_IC^2_J\times C^2_R$ matrix $\mathcal C_{2}(\mathbf A)\odot \mathcal C_{2}(\mathbf B)$ does not exceed its number of rows. In this subsection we draw conclusions for the generic case from the more general  Proposition \ref{Prop:1.16} and Corollary \ref{ProFullUniq1onematrixcor1}.

As in \cite{DeLathauwer2006, Psycho2006}, our proofs are based on the following lemma.

\begin{lemma}\label{Lemmaanalituni}
Let $f(\mathbf x)$ be an analytic function of  $\mathbf x \in\mathbb F^n$ and let $\mu_n$ be the Lebesgue measure on $\mathbb F^n$.
If $\mu_n\{\mathbf x:\ f(\mathbf x)=0\}>0$, then $f\equiv 0$.
\end{lemma}
\begin{proof}
The result easily follows from the uniqueness theorem for analytic functions (see for instance \cite[Lemma 2, p. 1855]{AlmostSure2001}).
\end{proof}

The following corollary trivially follows from Lemma \ref{Lemmaanalituni}.
\begin{corollary}\label{corranalituni}
Let $f(\mathbf x)$ be an analytic function of  $\mathbf x \in\mathbb F^n$ and let $\mu_n$ be the Lebesgue measure on $\mathbb F^n$.
Assume that there exists a point $\mathbf x_0$ such that
$f(\mathbf x_0)\ne 0$. Then $\mu_n\{\mathbf x:\ f(\mathbf x)=0\}=0$.
\end{corollary}

We will use the following matrix analogue of Corollary \ref{corranalituni}.
\begin{lemma}\label{analogcorranalituni}
Let $\mathbf F(\mathbf x)=(f_{pq}(\mathbf x))_{p,q=1}^{P,Q}$, with $P\geq Q$, be an analytic matrix-valued function of
$\mathbf x\in\mathbb F^n$ (that is, each entry $f_{pq}(\mathbf x)$ is an analytic function of $\mathbf x$) and let $\mu_n$ be the Lebesgue measure on $\mathbb F^n$. Assume that there exists a point $\mathbf x_0$ such that
$\mathbf F(\mathbf x_0)$ has full column rank. Then
$$
\mu_n\{\mathbf x:\ \mathbf F(\mathbf x)\ \text{does not have full column rank} \}=0.
$$
\end{lemma}
{\em Proof.}
Let  $\mathbf f (\mathbf x):=\mathcal C_Q(\mathbf F(\mathbf x))$ and $L:=C^Q_P$.
Then $\mathbf f:\ \mathbb F^n  \rightarrow \mathbb F^{L}: \mathbf x \rightarrow \mathbf f(\mathbf x)=\left[\begin{matrix}f_1(\mathbf x)&\dots&f_{L}(\mathbf x)\end{matrix}\right]^T$ is a vector-valued analytic function.
Note that $\mathbf f (\mathbf x)=\vzero$ if and only if the matrix $\mathbf F(\mathbf x)$ does not have full column rank.
 Since  $\mathbf f (\mathbf x_0)\ne\vzero$, there exists
$l_0\in\{1,\dots, L\}$ such that $f_{l_0} (\mathbf x)\ne 0$. Hence, by Corollary \ref{corranalituni}, $\mu_n\{ x:\  f_{l_0}(\mathbf x)=0\}=0$.
Therefore,
\begin{align*}
&\mu_n\{\mathbf x:\ \mathbf F(\mathbf x)\ \text{does not have full column rank} \}=\mu_n\{\mathbf x:\ \mathbf f(\mathbf x)=\vzero\}\\
=&\mu_n\left\{\bigcap\limits_{l=1}^{L}\{ x:\  f_l(\mathbf x)=0\}\right\}\leq\mu_n\{ x:\  f_{l_0}(\mathbf x)=0\}=0.\qquad\endproof
\end{align*}

The following lemma implies that, if $k_{\mathbf C}=r_{\mathbf C}$, then \eqref{eq5.1} in Proposition \ref{proponematrixisunique} holds generically, provided there exist matrices $\mathbf A_0\in\mathbb F^{I\times R}$ and $\mathbf B_0\in\mathbb F^{J\times R}$ for which $\mathcal C_{m_{\mathbf C}}(\mathbf A_0) \odot \mathcal C_{m_{\mathbf C}}(\mathbf B_0)$ has full column rank.
\begin{lemma}\label{Lemma:kc=rc}
Suppose the matrices $\mathbf A_0\in\mathbb F^{I\times R}$, $\mathbf B_0\in\mathbb F^{J\times R}$, and $\mathbf C\in\mathbb F^{K\times R}$ satisfy the following
conditions:
$$
k_{\mathbf A_0}=\min(I,R),\quad k_{\mathbf B_0}=\min(J,R),\quad k_{\mathbf C}=r_{\mathbf C}.
$$
Suppose further the matrix $\mathcal C_{m_{\mathbf C}}(\mathbf A_0)\odot \mathcal C_{m_{\mathbf C}}(\mathbf B_0)$ has full column rank, where $m=R-r_{\mathbf C}+2$.
Then
\begin{equation}
\max(\min(I,J-1,  R-1),\min(I-1,J, R-1))+k_{\mathbf C}\geq R+1.\label{eq:IJ}
\end{equation}
\end{lemma}
\begin{proof}
By Proposition \ref{Prop:KCUW} (2) and (9),
 $\min(k_{\mathbf A_0},k_{\mathbf B_0})\geq m_{\mathbf C}$. Hence,
$$
\min(I,J, R)\geq \min(k_{\mathbf A_0},k_{\mathbf B_0})\geq m_{\mathbf C}=R-r_{\mathbf C}+2=R-k_{\mathbf C}+2.
$$
Therefore,
\begin{equation*}
\begin{split}
\max(\min(I,J-1, R-1),\min(I-1,J,R-1))&+k_{\mathbf C}\geq \min(I-1,J-1)+k_{\mathbf C}\\
 &\geq R-k_{\mathbf C}+2-1+k_{\mathbf C}=R+1.
\end{split}
\end{equation*}
Hence, \eqref{eq:IJ} holds.
\end{proof}

The following proposition is the main result of this section.
\begin{proposition}\label{mostgeneralgeneric}
Let the matrix $\mathbf C\in\mathbb F^{K\times R}$ be fixed and suppose  $k_{\mathbf C}\geq 1$.
Assume that  there exist matrices $\mathbf A_0\in \mathbb F^{I\times R}$ and  $\mathbf B_0\in \mathbb F^{J\times R}$ such that
$\mathcal C_{m}(\mathbf A_0)\odot \mathcal C_{m}(\mathbf B_0)$  has full column rank, where $m=R-r_{\mathbf C}+2$.
Set $n=(I+J)R$.
Then
\begin{itemize}
\item[\textup{(i)}]
\begin{equation*}
\begin{split}
\mu_n\{(\mathbf A,\mathbf B):\ &\mathcal T:=[\mathbf A,\mathbf B,\mathbf C]_R \text{ has rank less than }R \text{ or}\\
&\text{the third factor matrix of }\mathcal T\text{ is not unique}\}=0.
\end{split}
\end{equation*}
\item[\textup{(ii)}]
 If additionally,  $k_{\mathbf C}=r_{\mathbf C}$, or \eqref{eq:IJ} holds, then
\begin{equation}\label{mu_n}
\begin{split}
\mu_n\{(\mathbf A,\mathbf B):\ \mathcal T:=[\mathbf A,\mathbf B,\mathbf C]_R &\text{ has rank less than }R \text{ or}\\
&\text{ the CPD of }\mathcal T\text{ is not unique}\}=0.
\end{split}
\end{equation}
\end{itemize}
\end{proposition}
{\em Proof.}
$\textup{(i)}$
Let $P:=C^m_IC^m_J$, $Q:=C^m_R$, $n:=(I+J)R$, $\mathbf x:=(\mathbf A,\mathbf B)$, $\mathbf x_0:=(\mathbf A_0,\mathbf B_0)$ and
$\mathbf F(\mathbf x):= \mathcal C_{m}(\mathbf A)\odot \mathcal C_{m}(\mathbf B)$.
Since $k_{\mathbf C}\geq 1$, from Proposition \ref{Prop:1.16} and Lemma \ref{analogcorranalituni} it follows that
\begin{equation*}
\begin{split}
\mu_n\{(\mathbf A,\mathbf B):\ &\mathcal T:=[\mathbf A,\mathbf B,\mathbf C]_R \text{ has rank less than }R \text{ or}\\
&\text{the third factor matrix of }\mathcal T\text{ is not unique}\}\\
\leq
\mu_n\{(\mathbf A,\mathbf B):\ &\mathcal C_{m}(\mathbf A)\odot \mathcal C_{m}(\mathbf B) \text{ does not have full column rank}\}=0.
\end{split}
\end{equation*}

$\textup{(ii)}$  By Lemma \ref{Lemma:kc=rc}, we can assume that \eqref{eq:IJ}  holds.
We obviously have
$$
\mu_n\{(\mathbf A,\mathbf B):\ k_{\mathbf A} < \min(I,  R)\ \text{ or } k_{\mathbf B}<\min( J, R)\}=0.
$$
Hence, by \eqref{eq:IJ},
$$
\mu_n\{(\mathbf A,\mathbf B):\ \text{\eqref{eq5.1} does not hold}\}=0.
$$
From Proposition \ref{proponematrixisunique} and \textup{(i)} it follows that
\begin{equation*}
\begin{split}
\mu_n\{(\mathbf A,\mathbf B):\ \mathcal T:=[\mathbf A,\mathbf B,\mathbf C]_R &\text{ has rank less than }R \text{ or}\\
&\text{ the CPD of }\mathcal T\text{ is not unique}\}\\
\leq
\mu_n\{(\mathbf A,\mathbf B):\ \mathcal T:=[\mathbf A,\mathbf B,\mathbf C]_R &\text{ has rank less than }R \text{ or}\\
&\text{the third factor matrix of }\mathcal T\text{ is not unique or} \\
&\text{\eqref{eq5.1} does not hold}\}=0.\qquad\endproof
\end{split}
\end{equation*}

\begin{proposition}\label{prop111overall}
The  CPD of  an $I\times J\times K$ tensor of rank $R$ is
generically unique if there exist matrices $\mathbf A_0\in \mathbb F^{I\times R}$ and  $\mathbf B_0\in \mathbb F^{J\times R}$ such that
$\mathcal C_{m}(\mathbf A_0)\odot \mathcal C_{m}(\mathbf B_0)$  has full column rank, where $m=R-\min (K,R)+2$.
\end{proposition}

{\em Proof.}
Generically we have $r_{\mathbf C}=\min(K,R)$. Let $N=(I+J+K)R$, $n=(I+J)R$, and let
$\Omega=\{\mathbf C:\ k_{\mathbf C}<r_{\mathbf C}\}\subset \mathbb F^{KR}$.  By  application of Lemma \ref{analogcorranalituni}, one obtains that $\mu_{KR}(\Omega)=0$.
From Proposition \ref{mostgeneralgeneric} it follows that  \eqref{mu_n} holds for $\mathbf C\not\in\Omega$.
Now
\begin{equation*}
\begin{split}
\mu_N\{(\mathbf A,\mathbf B,\mathbf C):\ \mathcal T:=[\mathbf A,\mathbf B,\mathbf C]_R &\text{ has rank less than }R \text{ or}\\
&\text{ the CPD of }\mathcal T\text{ is not unique}\}=0
\end{split}
\end{equation*}
follows from Fubini's theorem \cite[Theorem C, p. 148]{halmos1974measure}.
\qquad\endproof

{\em Proof of Proposition \ref{prop111overallbig}.}
Proposition \ref{prop111overallbig}  follows from Proposition \ref{prop111overall} by permuting factors.
%
%
\qquad\endproof

\subsection{Generic uniqueness of SFS-CPD}\label{generic:indscal}

For generic uniqueness of the SFS-CPD we resort to the following definition.

\begin{definition}\label{def1.7:indscal}
Let $\mu$ be the Lebesgue measure on $\mathbb F^{(2I+K)R}$.
The SFS-CPD  of an $I\times I\times K$ tensor of SFS-rank $R$ is $generically$ $unique$ if
$$
\mu\{(\mathbf A,\mathbf C):\ \text{the SFS-CPD }  \text{of the tensor }\ [\mathbf A,\mathbf A,\mathbf C]_R
\text{ is not unique }\}=0.
$$
\end{definition}
We have the following counterpart of Proposition \ref{prop111overallbig}.

\begin{proposition}\label{prop111overall:indscal}
The  SFS-CPD of  an $I\times I\times K$ SFS-tensor of SFS-rank $R$ is
generically unique if there exist matrices $\mathbf A_0\in \mathbb F^{I\times R}$ and  $\mathbf C_0\in \mathbb F^{K\times R}$ such that
 $\mathcal C_{m_{\mathbf C}}(\mathbf A_0)\odot \mathcal C_{m_{\mathbf C}}(\mathbf A_0)$ or
$\mathcal C_{m_{\mathbf A}}(\mathbf A_0)\odot \mathcal C_{m_{\mathbf A}}(\mathbf C_0)$
  has full column rank, where $m_{\mathbf C}=R-\min (K,R)+2$ and $m_{\mathbf A}=R-\min (I,R)+2$.
\end{proposition}
\begin{proof}
The proof is obtained by combining Proposition \ref{prop111overallbig} and Lemma \ref{Lemma:5.13}.
\end{proof}

\subsection{Examples}\label{sec:last2}

\begin{example}\label{ex:pert}
This example illustrates how one may adapt the approach in subsections  \ref{sec:last1} and \ref{generic:indscal} to particular types of structured factor matrices.

Let $\mathcal I_4$ be the $4\times 4\times 4$ tensor with ones on the main diagonal and zero off-diagonal entries and let $\mathcal T=\mathcal I_4+\mathbf a\circ\mathbf b\circ\mathbf c$ be a generic rank-$1$  perturbation of $\mathcal I_4$. Then $\mathcal T=[[\mathbf I_4\ \mathbf a], [\mathbf I_4\ \mathbf b], [\mathbf I_4\ \mathbf c]]_5$.
Since the $k$-ranks of all  factor matrices of $\mathcal T$ are equal to $4$, it follows from Kruskal's Theorem \ref{theoremKruskal} that
$r_{\mathcal T}=5$ and that the CPD of $\mathcal T$ is unique.

Let us now consider structured rank-$1$ perturbations $\bar{\mathbf a}\circ\bar{\mathbf b}\circ\bar{\mathbf c}$ that do not change the fourth vertical, third horizontal, and second frontal slice of $\mathcal I_4$.
The vectors $\bar{\mathbf a}$, $\bar{\mathbf b}$, and $\bar{\mathbf c}$ admit the following parameterizations
$$
\bar{\mathbf a}=\left[\begin{matrix}a_1&a_2&a_3&0\end{matrix}\right],\quad
\bar{\mathbf b}=\left[\begin{matrix}b_1&b_2&0&b_4\end{matrix}\right],\quad
\bar{\mathbf c}=\left[\begin{matrix}c_1&0&c_3&c_4\end{matrix}\right],
$$
with $a_i,b_j,c_k\in\mathbb F$.

Now the $k$-ranks of all  factor matrices of $\bar{\mathcal T}:=\mathcal I_4+\bar{\mathbf a}\circ\bar{\mathbf b}\circ\bar{\mathbf c}$ are equal to $3$, and (generic) uniqueness of the CPD of $\bar{\mathcal T}$ does not follow from Kruskal's Theorem \ref{theoremKruskal}.

We show that
$$
\mu_9\{(\bar{\mathbf a}, \bar{\mathbf b},\bar{\mathbf c}): \text{the CPD of }\bar{\mathcal T}:=\mathcal I_4+\bar{\mathbf a}\circ\bar{\mathbf b}\circ\bar{\mathbf c}\text{ is not unique or }r_{\bar{\mathcal T}}<5\}=0,
$$
that is, the CPD of rank-$1$ structured generic perturbation of $\mathcal I_4$ is again unique.

Let the matrices $\mathbf A_0$, $\mathbf B_0$, and $\mathbf C_0$ be given by the matrices $\mathbf A$, $\mathbf B$, and $\mathbf C$, respectively, in Example \ref{example5.7}. As in Example \ref{example5.7}
the matrix pairs $(\mathbf A_0,\mathbf B_0)$, $(\mathbf B_0,\mathbf C_0)$, and $(\mathbf C_0,\mathbf A_0)$ satisfy
condition $\text{\textup{(C{\scriptsize 3})}}$.
Then, by Lemma \ref{analogcorranalituni},
\begin{gather*}
\mu_6\{(\bar{\mathbf a}, \bar{\mathbf b}): \mathcal C_3(\mathbf A)\odot\mathcal C_3(\mathbf B)\text{ does not have full column rank}\}=0,\\
\mu_6\{(\bar{\mathbf b}, \bar{\mathbf c}): \mathcal C_3(\mathbf B)\odot\mathcal C_3(\mathbf C)\text{ does not have full column rank}\}=0.
\end{gather*}
By Fubini's theorem \cite[Theorem C, p. 148]{halmos1974measure},
\begin{gather*}
\mu_9\{(\bar{\mathbf a}, \bar{\mathbf b},\bar{\mathbf c}): \mathcal C_3(\mathbf A)\odot\mathcal C_3(\mathbf B)\text{ has not full column rank or } k_{\mathbf C}<3 \text{ or }r_{\mathbf C}<4  \}=0,\\
\mu_9\{(\bar{\mathbf a}, \bar{\mathbf b},\bar{\mathbf c}): \mathcal C_3(\mathbf B)\odot\mathcal C_3(\mathbf C)\text{ has not full column rank or } k_{\mathbf A}<3 \text{ or }r_{\mathbf A}<4  \}=0.
\end{gather*}
Now generic uniqueness of the structured rank-$1$ perturbation of $\mathcal I_4$ follows from Proposition \ref{proptwomatrixisunique}.
\end{example}

\begin{example}\label{example4.7}
Let $\mathcal T=[\mathbf A, \mathbf B, \mathbf C]_R$ denote a PD of an $I\times I\times (2I-1)$ tensor, where $I\geq 4$.
Generically, $k_{\mathbf A} = k_{\mathbf B} = I$ and $k_{\mathbf C} = 2I-1$.
Then Kruskal's condition \eqref{Kruskal} guarantees generic uniqueness for $R\leq\lfloor\frac{I+I+2I-1-2}{2}\rfloor=2I-1$.

On the other hand,  (\ref{C2generic}) guarantees generic uniqueness of the CPD under the conditions $R \leq 2I-1$ and $C^2_R \leq (C^2_I)^2$.  The maximum value of $R$ that satisfies these bounds is shown in the column corresponding to $m_{\mathbf C}=2$ in Table \ref{tab2}. The condition in Theorem \ref{th:fcr} is even more relaxed.

We now move to cases where $R > 2I-1$, where Theorem \ref{th:fcr} no longer applies.  By Proposition \ref{prop111overall},
the CPD of $\mathcal T$ of an $I\times I\times (2I-1)$ tensor of rank $R$ is generically unique if there exist matrices $\mathbf A_0\in \mathbb F^{I\times R}$ and  $\mathbf B_0\in \mathbb F^{I\times R}$ such that $\mathcal C_{m_{\mathbf C}}(\mathbf A_0)\odot \mathcal C_{m_{\mathbf C}}(\mathbf B_0)$ has full column rank, where $m_{\mathbf C}=R-(2I-1)+2=R-2I+3$. The proof of Proposition \ref{prop111overall} shows that, if there exist $\mathbf A_0$ and  $\mathbf B_0$ such that $\mathcal C_{m_{\mathbf C}}(\mathbf A_0)\odot \mathcal C_{m_{\mathbf C}}(\mathbf B_0)$ has full column rank, then actually $\mathcal C_{m_{\mathbf C}}(\mathbf A_0)\odot \mathcal C_{m_{\mathbf C}}(\mathbf B_0)$ has full column rank with probability one when $\mathbf A_0$ and  $\mathbf B_0$ are drawn from continuous distributions. Hence, we generate random $\mathbf A_0$ and  $\mathbf B_0$ and check up to which value of $R$ the matrix $\mathcal C_{m_{\mathbf C}}(\mathbf A_0)\odot \mathcal C_{m_{\mathbf C}}(\mathbf B_0)$ has full column rank. Table \ref{tab2} shows the results for $4 \leq I \leq 9$. For instance, we obtain that the  CPD of a $9\times 9\times 17$ tensor  of rank $R$ is generically unique if $R \leq 20$.
(By of comparison, Theorem  \ref{th:fcr} only guarantees uniqueness up to $R=17$.)

Proposition \ref{prop111overall} corresponds to condition (i) in Proposition \ref{prop111overallbig}. Note that, for $R \geq 2I-1$, we generically have $m_{\mathbf A}=m_{\mathbf B}=R-I+2\geq I+1$ such that the $m_{\mathbf B}$-th compound matrix of ${\mathbf A}$ and the $m_{\mathbf A}$-th compound matrix of ${\mathbf B}$ are not defined. Hence, we cannot resort to condition (ii) or (iii) in Proposition \ref{prop111overallbig}.

\begin{table}[htbp]
\caption{Upper bounds on $R$ under which generic uniqueness of the CPD of an $I\times I\times (2I-1)$ tensor is guaranteed by Proposition  \ref{prop111overall}.}
\begin{center}\footnotesize
\begin{tabular}{|c|c|c|c|c|}
  \hline
   dimensions of  $\mathcal T$ & \multicolumn{4}{|c|}{$m=R-2I+3$} \\
\cline{2-5}
$I\times I\times (2I-1)$       &  2&3&4&5\\
\hline
  $4\times 4\times 7$    & $7$  &     &      & \\
  $5\times 5\times 9$    & $9$  &     &      & \\
  $6\times 6\times 11$   & $11$ & $12$&      & \\
  $7\times 7\times 13$   & $13$ & $14$&      & \\
  $8\times 8\times 15$   & $15$ & $16$& $17$ & \\
  $9\times 9\times 17$   & $17$ & $18$& $19$ & $20$\\
     \hline
\end{tabular}
\end{center}
\label{tab2}
\end{table}
\end{example}
\begin{remark}
For $I=3$ and $R=2I-1=5$, the CPD of an $I\times I\times (2I-1)$ tensor $\mathcal T$ is not generically unique \cite{tenBerge2004},\cite{rankeqkrank2006}.
This is the reason why in Table \ref{tab2} we start from $I=4$.
\end{remark}

\begin{remark}
It was shown in \cite[Corollary 1, p.1852]{AlmostSure2001} that the matrix $\mathcal C_1(\mathbf A)\odot\mathcal C_1(\mathbf B)=\mathbf A\odot\mathbf B$ has full column rank with probability one
when the number of rows of $\mathbf A\odot\mathbf B$ does not exceed its  number of columns.
The same statement was made for the matrix $\mathcal C_2(\mathbf A)\odot \mathcal C_2(\mathbf B)$ in \cite{DeLathauwer2006, Psycho2006}.
However, the statement does not hold for compound matrices of arbitrary order. For instance, it does not hold for
$\mathcal C_5(\mathbf A)\odot \mathcal C_5(\mathbf B)$,  where $\mathbf A\in \mathbb F^{6\times 9}$ and $\mathbf B\in \mathbb F^{7\times 9}$.
\end{remark}
\begin{example}\label{exampleremark4.9}
Let  $\mathcal T=[\mathbf A, \mathbf B, \mathbf C]_9$ denote a generic PD in 9 terms in the $6\times7\times 6$ case.
Then $m_{\mathbf A}=m_{\mathbf C}=9-6+2=5$ and $m_{\mathbf B}=9-7+2=4$. The matrices
$\mathbf M_{\mathbf C}:=\mathcal C_5(\mathbf A)\odot \mathcal C_5(\mathbf B)$ and
$\mathbf M_{\mathbf A}:=\mathcal C_5(\mathbf B)\odot \mathcal C_5(\mathbf C)$ have $C^5_6C^5_7=C^5_9=126$ rows and columns.
Numerical experiments  indicate that $\dim \textup{ker}(\mathbf M_{\mathbf C})=\dim \textup{ker}(\mathbf M_{\mathbf A})=15$ with probability one.
Hence, we cannot use Proposition \ref{prop111overallbig} \textup{(i)} or \textup{(ii)}  for proving uniqueness of the CPD. On the other hand, the $C^4_6C^4_6\times C^4_9$
($225\times 126$) matrix $\mathbf M_{\mathbf B}:=\mathcal C_4(\mathbf C)\odot \mathcal C_4(\mathbf A)$ turns out to have full column rank for a random choice of $\mathbf A$  and $\mathbf C$. Hence, by Proposition \ref{prop111overallbig} \textup{(iii)}, the CPD is generically unique.
\end{example}
\begin{example} Here we consider $I\times I\times K$ tensors with   $I\in\{4,\dots,9\}$ and $K\in\{2,\dots,33\}$, which is more general than Example \ref{example4.7}.

We check up to which value of $R$ one of the conditions in Proposition \ref{prop111overallbig} holds for a random choice of the factor matrices. Up to this value the CPD is generically unique. The results are shown as the left-most values in Table \ref{tab3}. We also check up to which value of $R$  one of the conditions in Proposition \ref{prop111overall:indscal} holds for a random choice of the factor matrices. Up to this value the SFS-CPD is generically unique. The results are shown as the middle values in Table \ref{tab3}.

The right-most values correspond to the maximum value of $R$ for which generic uniqueness is guaranteed by
Kruskal's Theorems \ref{theoremKruskal}--\ref{theoremKruskalnew2}, i.e., the largest value of $R$ that satisfies
$
2\min(I,R)+\min(K,R)\geq 2R+2.
$
Note that Kruskal's bound is the same for CPD and SFS-CPD. The bold values in the table correspond to the results that were not yet covered  by Kruskal's Theorems \ref{theoremKruskal}--\ref{theoremKruskalnew2} or Proposition \ref{K2C2U2unique} ($m=2$).

\begin{table}
\caption{Upper bounds on $R$ under which generic uniqueness of the CPD (left and right value)  and SFS-CPD (middle and right value) of an $I\times I\times K$ tensor is guaranteed by Proposition \ref{prop111overallbig} (left), Proposition  \ref{prop111overall:indscal} (middle), and Kruskal's Theorems \ref{theoremKruskal}--\ref{theoremKruskalnew2} (right).
The values shown in bold correspond to the results that were not yet covered  by Kruskal's Theorems \ref{theoremKruskal}--\ref{theoremKruskalnew2} or Proposition \ref{K2C2U2unique} ($m=2$).}
\begin{center}\footnotesize
\begin{tabular}{|c|c|rrrrrr|}
\cline{3-8}
\multicolumn{2}{c}{}& \multicolumn{6}{|c|}{$I$}\\
\cline{3-8}
\multicolumn{2}{c|}{} &4 &5 &6 &7 &8 &9 \\
\cline{1-8}
\multirow{32}{*}{$K$}
&2  &4, 4, 4       &5, 5, 5       &6, 6, 6       &7, 7, 7       &8, 8, 8       &9, 9, 9\\
&3  &4, 4, 4       &5, 5, 5       &6, 6, 6       &7, 7, 7       &8, 8, 8       &9, 9, 9\\
&4  &5, 5, 5       &6, 6, 6       &7, 7, 7       &8, 8, 8       &9, 9, 9       &10, 10, 10\\
&5  &5, 5, 5       &6, 6, 6       &7, 7, 7       &8, 8, 8       &{\bf 10}, {\bf 10}, 9    &{\bf 11}, {\bf 11}, 10\\
&6  &6, 6, 6       &7, 7, 7       &8, 8, 8       &9, 9, 9       &10, 10, 10   &11, 11, 11\\
&7  &7, 6, 6       &{\bf 8}, 7, 7       &{\bf 9}, 8, 8       &9, 9, 9       &{\bf 11}, {\bf 11}, 10   &{\bf 12}, {\bf 12}, 11\\
&8  &8, 6, 6       &{\bf 9}, 8, 8       &9, 9, 9       &10, 10, 10   &11, 11, 11   &12, 12, 12\\
&9  &9, 6, 6       &9, 9, 8       &{\bf 10}, {\bf 10}, 9    &{\bf 11}, 10, 10   &{\bf 12}, 11, 11   &{\bf 13}, {\bf 13}, 12\\
&10 &9, 6, 6       &10, 10, 8    &{\bf 11}, 10, 10   &{\bf 12}, 11, 11   &{\bf 13}, 12, 12   &{\bf 14}, 13, 13\\
&11 &9, 6, 6       &11, 10, 8    &{\bf 12}, 11, 10   &{\bf 13}, {\bf 12}, 11   &{\bf 14}, {\bf 13}, 12   &{\bf 15}, {\bf 14}, 13\\
&12 &9, 6, 6       &12, 10, 8    &{\bf 13}, 12, 10   &{\bf 14}, {\bf 13}, 12   &{\bf 15}, {\bf 14}, 13   &{\bf 15}, {\bf 15}, 14\\
&13 &9, 6, 6       &13, 10, 8    &{\bf 14}, 13, 10   &{\bf 14}, {\bf 14}, 12   &{\bf 15}, {\bf 15}, 13   &{\bf 16}, {\bf 15}, 14\\
&14 &9, 6, 6       &14, 10, 8    &14, 14, 10   &{\bf 15}, {\bf 15}, 12   &{\bf 16}, {\bf 15}, 14   &{\bf 17}, {\bf 16}, 15\\
&15 &9, 6, 6       &14, 10, 8    &15, 15, 10   &{\bf 16}, 15, 12   &{\bf 17}, {\bf 16}, 14   &{\bf 18}, {\bf 17}, 15\\
&16 &9, 6, 6       &14, 10, 8    &16, 15, 10   &{\bf 17}, 16, 12   &{\bf 18}, {\bf 17}, 14   &{\bf 19}, {\bf 18}, 16\\
&17 &9, 6, 6       &14, 10, 8    &17, 15, 10   &{\bf 18}, 17, 12   &{\bf 19}, {\bf 18}, 14   &{\bf 20}, {\bf 19}, 16\\
&18 &9, 6, 6       &14, 10, 8    &18, 15, 10   &{\bf 19}, 18, 12   &{\bf 20}, {\bf 19}, 14   &{\bf 20}, {\bf 20}, 16\\
&19 &9, 6, 6       &14, 10, 8    &19, 15, 10   &{\bf 20}, 19, 12   &{\bf 20}, {\bf 20}, 14   &{\bf 21}, {\bf 20}, 16\\
&20 &9, 6, 6       &14, 10, 8    &20, 15, 10   &20, 20, 12   &{\bf 21}, 20, 14   &{\bf 22}, {\bf 21}, 16\\
&21 &9, 6, 6       &14, 10, 8    &21, 15, 10   &21, 20, 12   &{\bf 22}, 21, 14   &{\bf 23}, {\bf 22}, 16\\
&22 &9, 6, 6       &14, 10, 8    &21, 15, 10   &22, 20, 12   &{\bf 23}, 22, 14   &{\bf 24}, {\bf 23}, 16\\
&23 &9, 6, 6       &14, 10, 8    &21, 15, 10   &23, 20, 12   &{\bf 24}, 23, 14   &{\bf 25}, {\bf 24}, 16\\
&24 &9, 6, 6       &14, 10, 8    &21, 15, 10   &24, 20, 12   &{\bf 25}, 24, 14   &{\bf 26}, {\bf 25}, 16\\
&25 &9, 6, 6       &14, 10, 8    &21, 15, 10   &25, 20, 12   &{\bf 26}, 25, 14   &{\bf 26}, 25, 16\\
&26 &9, 6, 6       &14, 10, 8    &21, 15, 10   &26, 20, 12   &{\bf 27}, 26, 14   &{\bf 27}, 26, 16\\
&27 &9, 6, 6       &14, 10, 8    &21, 15, 10   &27, 20, 12   &27, 26, 14   &{\bf 28}, 27, 16\\
&28 &9, 6, 6       &14, 10, 8    &21, 15, 10   &28, 20, 12   &28, 26, 14   &{\bf 29}, 28, 16\\
&29 &9, 6, 6       &14, 10, 8    &21, 15, 10   &29, 20, 12   &29, 26, 14   &{\bf 30}, 29, 16\\
&30 &9, 6, 6       &14, 10, 8    &21, 15, 10   &30, 20, 12   &30, 26, 14   &{\bf 31}, 30, 16\\
&31 &9, 6, 6       &14, 10, 8    &21, 15, 10   &30, 20, 12   &31, 26, 14   &{\bf 32}, 31, 16\\
&32 &9, 6, 6       &14, 10, 8    &21, 15, 10   &30, 20, 12   &32, 26, 14   &{\bf 33}, 32, 16\\
&33 &9, 6, 6       &14, 10, 8    &21, 15, 10   &30, 20, 12   &33, 26, 14   &{\bf 34}, 33, 16\\
\hline
\end{tabular}
\end{center}
\label{tab3}
\end{table}
\end{example}
\begin{remark}
Most of the improved left-most values in Table \ref{tab3}
also follow from Theorems \ref{th:gen0}, \ref{th:fcr}--\ref{cubic}. (Concerning the latter, if the CPD of an $I\times I\times I$  tensor of rank $R$ is generically unique for $R\leq k(I)$, then a forteriori the CPD of a rank-$R$ $I\times I\times K$  tensor with $K > I$ is generically unique for $R\leq k(I)$.) An important difference is that our bounds remain valid for many constrained CPDs. We briefly give two examples. Rather than going into details, let us suffice by mentioning that (generic) uniqueness in these examples may be defined and studied in the same way as it was done in Subsections \ref{sec:last1} and \ref{generic:indscal} for unsymmetric CPD and SFS-CPD, respectively.
\begin{enumerate}
\item
Let the third factor matrix of $I\times I\times K$ tensor $\mathcal T$
belong to a class of structured matrices $\Omega$  such that the condition $I+k_{\mathbf C}\geq R+2$ is valid for generic $\mathbf C\in\Omega$.
An example of a class for which this may be true, is the class of
$K\times R$ Hankel matrices.
In Subsection \ref{sec:last1} Proposition  \ref{mostgeneralgeneric} leads to Proposition \ref{prop111overallbig} for unconstrained CPD. Similarly,
Proposition \ref{mostgeneralgeneric} with condition   \eqref{eq:IJ} replaced by condition  $I+k_{\mathbf C}\geq R+2$
leads to an analogue of Proposition \ref{prop111overallbig} that guarantees that
 a  CPD with the third factor matrix belonging to $\Omega$ is generically unique for $R$ bounded by the values in Table \ref{tab3}
 (left values for unconstrained first and second factor matrices, and middle values in the case of partial symmetry).
 \item
Let us now assume that the third factor matrix is unstructured and that the first two matrices have Toeplitz structure. Random Toeplitz matrices also yield the values in Table \ref{tab3}. Hence, such a constrained CPD is again generically unique for $R$ bounded by the values in Table \ref{tab3}.
\end{enumerate}
\end{remark}
\begin{remark}\label{problastremark}
In the case $r_{\mathbf C}=R$, both \cond{C}{2} and \cond{U}{2} are sufficient for overall CPD uniqueness, see \eqref{eq:1.11}.
In the case of \cond{C}{2}, we generically have condition \eqref{C2generic}. The more relaxed generic condition derived from \cond{U}{2}
is given in Theorem \ref{th:fcr}. For the case $r_{\mathbf C}<R$ we have  obtained the deterministic result in Corollary \ref{ProFullUniq1onematrixcor1} and its its generic version Proposition \ref{prop111overallbig}, both based on condition  \cond{C}{m}. This suggests that by starting from Corollary \ref{ProFullUniq1onematrix}, based on  \cond{U}{m}, more relaxed generic uniqueness results may be obtained.

On the other hand, in Example \ref{ex:77710}
we have studied CPD of a rank-10 ($7\times 7\times 7$) tensor. Simulations along the lines of Example \ref{ex:77710}
suggest that condition \cond{W}{5} holds for random factor matrices, which then implies generic overall CPD uniqueness for $R=10$. Starting from
\cond{C}{5} we have only demonstrated generic uniqueness up to $R=9$, see the entry for $I=K=7$ in Table \ref{tab3}.
This suggests that by starting from Proposition \ref{ProFullUniq1onematrixWm},
based on \cond{W}{m}, further relaxed generic uniqueness results may be obtained.
\end{remark}

\section{Conclusion}

Using results obtained in Part I \cite{PartI}, we have obtained new conditions guaranteeing uniqueness of a CPD. In the framework of the new uniqueness theorems, Kruskal's theorem and the existing uniqueness theorems for the case $R = r_{\mathbf C}$ are special cases. We have derived both deterministic and generic conditions.

The results can be easily adapted to the case of PDs in which one or several factor matrices are equal, such as INDSCAL. In the deterministic conditions the equalities can simply be substituted. In the generic setting one checks the same rank constraints as in the unconstrained case for a random example. The difference is that there are fewer independent entries to draw randomly. This may decrease the value of $R$ up to which uniqueness is guaranteed. However, the procedure for determining this maximal value is completely analogous. The same holds true for PDs in which one or several factor matrices have structure (Toeplitz, Hankel, Vandermonde, etc.).
\section{Acknowledgments}
 The authors would like to thank the anonymous reviewers
for their valuable comments and their suggestions to improve
the presentation of the paper. The authors are  also grateful for useful  suggestions from
Professor A. Stegeman (University of Groningen, The Netherlands).

\bibliographystyle{siam}
\bibliography{PartII}

\begin{thebibliography}{10}

\bibitem{1970_Carroll_Chang}
{\sc J.~Carroll and J.-J. Chang}, {\em {A}nalysis of individual differences in
  multidimensional scaling via an {N}-way generalization of "{E}ckart-{Y}oung"
  decomposition}, Psychometrika, 35 (1970), pp.~283--319.

\bibitem{ChiantiniandOttaviani}
{\sc L.~Chiantini and G.~Ottaviani}, {\em On generic identifiability of
  3-tensors of small rank}, SIAM J. Matrix Anal. Appl., 33 (2012),
  pp.~1018--1037.

\bibitem{2009Comonetall}
{\sc P.~Comon, X.~Luciani, and A.~L.~F. de~Almeida}, {\em Tensor
  decompositions, alternating least squares and other tales}, J. Chemometrics,
  23 (2009), pp.~393--405.

\bibitem{DeLathauwer2006}
{\sc L.~De~Lathauwer}, {\em A {L}ink {B}etween the {C}anonical {D}ecomposition
  in {M}ultilinear {A}lgebra and {S}imultaneous {M}atrix {D}iagonalization},
  SIAM J. Matrix Anal. Appl., 28 (2006), pp.~642--666.

\bibitem{LievenLL1}
\leavevmode\vrule height 2pt depth -1.6pt width 23pt, {\em Blind separation of
  exponential polynomials and the decomposition of a tensor in
  rank--$({L}_r,{L}_r,1)$ terms}, SIAM J. Matrix Anal. Appl., 32 (2011),
  pp.~1451--1474.

\bibitem{Lieven_ISPA}
\leavevmode\vrule height 2pt depth -1.6pt width 23pt, {\em A short introduction
  to tensor-based methods for {F}actor {A}nalysis and {B}lind {S}ource
  {S}eparation}, in ISPA 2011: Proceedings of the 7th International Symposium
  on Image and Signal Processing and Analysis,  (2011), pp.~558--563.

\bibitem{PartI}
{\sc I.~Domanov and L.~De~Lathauwer}, {\em {O}n the {U}niqueness of the
  {C}anonical {P}olyadic {D}ecomposition of third-order tensors --- {P}art {I}:
  {B}asic {R}esults and {U}niqueness of {O}ne {F}actor {M}atrix}, ESAT-SISTA
  Internal Report, 12-66, Leuven, Belgium: Department of Electrical Engineering
  (ESAT), KU Leuven,  (2012).

\bibitem{GuoMironBrieStegeman}
{\sc X.~Guo, S.~Miron, D.~Brie, and A.~Stegeman}, {\em {U}ni-{M}ode and
  {P}artial {U}niqueness {C}onditions for {CANDECOMP}/{PARAFAC} of
  {T}hree-{W}ay {A}rrays with {L}inearly {D}ependent {L}oadings}, SIAM J.
  Matrix Anal. Appl., 33 (2012), pp.~111--129.

\bibitem{halmos1974measure}
{\sc P.~R. Halmos}, {\em Measure theory}, Springer-Verlag, New-York, 1974.

\bibitem{JiangSid2004}
{\sc T.~Jiang and N.~D. Sidiropoulos}, {\em {K}ruskal's {P}ermutation {L}emma
  and the {I}dentification of {CANDECOMP}/{PARAFAC} and {B}ilinear {M}odels
  with {C}onstant {M}odulus {C}onstraints}, IEEE Trans. Signal Process., 52
  (2004), pp.~2625--2636.

\bibitem{AlmostSure2001}
{\sc T.~Jiang, N.~D. Sidiropoulos, and J.~M.~F. Ten~Berge}, {\em
  {A}lmost-{S}ure {I}dentifiability of {M}ultidimensional {H}armonic
  {R}etrieval}, IEEE Trans. Signal Process., 49 (2001), pp.~1849--1859.

\bibitem{KoldaReview}
{\sc T.~G. Kolda and B.~W. Bader}, {\em {T}ensor {D}ecompositions and
  {A}pplications}, SIAM Review, 51 (2009), pp.~455--500.

\bibitem{Krijnen1991}
{\sc W.~P. Krijnen}, {\em The analysis of three-way arrays by constrained
  {P}arafac methods}, DSWO Press, Leiden, 1991.

\bibitem{Kruskal1977}
{\sc J.~B. Kruskal}, {\em Three-way arrays: rank and uniqueness of trilinear
  decompositions, with application to arithmetic complexity and statistics},
  Linear Algebra Appl., 18 (1977), pp.~95--138.

\bibitem{Stegeman2009}
{\sc A.~Stegeman}, {\em On uniqueness conditions for {C}andecomp/{P}arafac and
  {I}ndscal with full column rank in one mode}, Linear Algebra Appl., 431
  (2009), pp.~211--227.

\bibitem{Psycho2006}
{\sc A.~Stegeman, J.~Ten~Berge, and L.~De~Lathauwer}, {\em Sufficient
  conditions for uniqueness in {C}andecomp/{P}arafac and {I}ndscal with random
  component matrices}, Psychometrika, 71 (2006), pp.~219--229.

\bibitem{rankeqkrank2006}
{\sc A.~Stegeman and J.~M.~F. Ten~Berge}, {\em Kruskal's condition for
  uniqueness in {C}andecomp/{P}arafac when ranks and $k$-ranks coincide},
  Comput. Stat. Data Anal., 50 (2006), pp.~210--220.

\bibitem{Strassen1983}
{\sc V.~Strassen}, {\em Rank and optimal computation of generic tensors},
  Linear Algebra Appl., 52--53 (1983), pp.~645--685.

\bibitem{tenBerge2004}
{\sc J.~M.~F. Ten~Berge}, {\em Partial uniqueness in {CANDECOMP}/{PARAFAC}}, J.
  Chemometrics, 18 (2004), pp.~12--16.

\bibitem{tenBergSidRoccie2004}
{\sc J.~M.~F. Ten~Berge, N.~D. Sidiropoulos, and R.~Rocci}, {\em Typical rank
  and indscal dimensionality for symmetric three-way arrays of order ${I}\times
  2\times 2$ or ${I}\times 3\times 3$}, Linear Algebra Appl., 388 (2004),
  pp.~363 -- 377.

\bibitem{LiuSid2001}
{\sc L.~Xiangqian and N.~D. Sidiropoulos}, {\em {C}ramer-{R}ao lower bounds for
  low-rank decomposition of multidimensional arrays}, IEEE Trans. Signal
  Process., 49 (2001), pp.~2074--2086.

\end{thebibliography}

\end{document}